\documentclass[a4paper,11pt]{scrartcl}
\usepackage{amsmath,amssymb}
\usepackage{amsthm}
\usepackage{graphicx}
\usepackage[square,sort&compress,numbers]{natbib}
\usepackage{url}
\usepackage{type1cm}
\usepackage{mathtools}\mathtoolsset{showonlyrefs=false}
\usepackage{booktabs}
\usepackage{subfig}

\newcommand{\Authornote}{\renewcommand{\thefootnote}{\fnsymbol{footnote}}}
\newcommand{\authornote}{\Authornote\footnote}

\theoremstyle{plain}
\newtheorem{theorem}{Theorem}[section]

\newtheorem{lemma}{Lemma}[section]

\theoremstyle{definition}

\newtheorem{example}{Example}[section]
\theoremstyle{remark}

\newcommand{\refex}[1]{Example~\ref{#1}}

\newcommand{\refthm}[1]{Theorem~\ref{#1}}
\newcommand{\reflmm}[1]{Lemma~\ref{#1}}

\newcommand{\reffig}[1]{Figure~\ref{#1}}
\newcommand{\reftab}[1]{Table~\ref{#1}}

\newcommand{\finbox}{\nolinebreak\hfill{\small $\blacksquare$}}

\newcommand{\MIN}{\mathop{\mathrm{Minimize}}}

\newcommand{\st}{\mathop{\mathrm{s.{\,}t.}}}
\newcommand{\ST}{\mathop{\mathrm{subject~to}}}

\newcommand{\tr}{\mathop{\mathrm{tr}}\nolimits}

\renewcommand{\Re}{\ensuremath{\mathbb{R}}}

\newcommand{\bi}[1]{\ensuremath{\boldsymbol{#1}}}

\newcommand{\rr}[1]{\ensuremath{\mathrm{#1}}}

\newcommand{\pdif}[2]{\frac{\partial #1}{\partial #2}}

\newcommand{\NC}{\ensuremath{\mathsf{N}}}
\newcommand{\PC}{\ensuremath{\mathcal{P}}}
\newcommand{\SC}{\ensuremath{\mathcal{S}}}

\begin{document}

\begin{center}
  {\Large\bfseries\sffamily%
  Structural Reliability under Uncertainty in Moments: }\\
  \medskip
  {\Large\bfseries\sffamily%
  Distributionally-robust Reliability-based Design Optimization }%
  \par%
  \bigskip%
  {%
  Yoshihiro Kanno~\authornote[2]{%
    Mathematics and Informatics Center, 
    The University of Tokyo, 
    Hongo 7-3-1, Tokyo 113-8656, Japan.
    E-mail: \texttt{kanno@mist.i.u-tokyo.ac.jp}. 
    }
  }
\end{center}

\begin{abstract}
  This paper considers structural optimization under a reliability 
  constraint, where the input distribution is only partially known. 
  Specifically, when we only know that the expected value vector and the 
  variance-covariance matrix of the input distribution belong to a given 
  convex set, we require that, for any realization of the input 
  distribution, the failure probability of a structure should be no 
  greater than a specified target value. 
  We show that this distributionally-robust reliability constraint can 
  be reduced equivalently to deterministic constraints. 
  By using this reduction, we can treat a reliability-based design 
  optimization problem under the distributionally-robust reliability 
  constraint within the framework of 
  deterministic optimization, specifically, nonlinear semidefinite 
  programming. 
  Two numerical examples are solved to show relation between 
  the optimal value and either the target reliability or the uncertainty 
  magnitude. 
\end{abstract}

\begin{quote}
  \textbf{Keywords}
  \par
  Reliability-based design optimization; 
  uncertain input distribution; 
  worst-case reliability; 
  robust optimization; 
  semidefinite programming; 
  duality. 
\end{quote}

\section{Introduction}

Reliability-based design optimization (RBDO) is a crucial tool for 
structural design in the presence of 
uncertainty \citep{VS10,AC10,YCLvTG11,MS19}. 
It adopts a probabilistic model of uncertainty, and evaluates 
the probability that a structural design satisfies (or, equivalently, 
fails to satisfy) performance requirements. 
An underlying premise is that complete knowledge on statistical 
information of uncertain parameters is available. 
In practice, however, it is often difficult to obtain statistical 
information with sufficient accuracy. 
This incents recent intensive study of RBDO with incomplete statistical 
information \citep{CCGLLG16,CAW10,GP06,IKK18,IK16,JCFY13,JCL19,%
MCCGLG18,NCLGL11,NCLGL11b,WHYWG20,YW08,ZM17,ZGXLE20}. 

Another methodology dealing with uncertainty in structural design is 
{\em robust design optimization\/} \citep{BS07,HZFY19,Kan20b}. 
Although there exist several different concepts in robust design 
optimization, in this paper we focus attention on the worst-case 
optimization methodology, which is called {\em robust optimization\/} 
in mathematical optimization community \citep{BtEN09}. 
This methodology adopts a possibilistic model of uncertainty, i.e., 
specifies the set of possible values that the uncertain parameters can 
take. 
We call this set an {\em uncertainty set\/}. 
Then, the objective value in the worst case is optimized, under the 
condition that the constraints are satisfied in the worst cases. 

This paper deals with RBDO when the input distribution is only partially 
known. 
Specifically, we assume that the true expected value vector and 
the true variance-covariance matrix are unknown (i.e., the true values 
of the first two moments of the input distribution are unknown), but 
they are known to belong to a given closed convex set. 
For example, suppose that the input distribution is a normal 
distribution, and we only know that each component of the expected value 
vector and the variance-covariance matrix belongs to a given closed 
interval. 
Then, for each possible realization of pair of the expected value 
vector and the variance-covariance matrix, there exists a single 
corresponding normal distribution. 
The set of all such normal distributions is considered as an 
uncertainty set of the input distribution.\footnote{%
This uncertainty set is dealt with in section~\ref{sec:robust}. } 
As another example, suppose that distribution type of the input 
distribution is also unknown. 
Then the uncertainty set is the set of all probability distributions, 
the expected value vector and the variance-covariance matrix of which 
belong to a given set.\footnote{%
This uncertainty set is dealt with in section~\ref{sec:general.non-Gaussian}. } 
Among probability distributions belonging to a specified uncertainty 
set defined as above, the worst-case distribution is the one with which 
the failure probability takes the maximum value. 
Our methodology is that we require a structure to satisfy the 
reliability constraint evaluated with the worst-case distribution. 
In other words, for {\em any\/} probability distribution belonging to 
the uncertainty set, the failure probability should be no greater than a 
specified target value. 
Thus, the methodology guarantees robustness of the structural 
reliability against uncertainty in the input distribution.\footnote{%
More precisely, the uncertainty here means the uncertainty in the 
expected value vector and the variance-covariance matrix of the input 
distribution. }
The major contribution of this paper is to show, under some assumptions, 
that this structural requirement is equivalently converted to a form of 
constraints that can be treated in conventional deterministic optimization. 
As a result, a design optimization problem under this structural 
requirement can be solved with a deterministic nonlinear 
optimization approach. 

Recently, RBDO methods with uncertainty in the input distribution have 
received considerable attention, because in practice it is often that 
the number of available samples of random variables is insufficient. 
For example, \citet{GP06} and \citet{YW08} proposed Bayesian approaches 
to compute the confidence that a structural design satisfies a target 
reliability constraint, when both a finite number of samples and 
probability distributions of uncertain parameters are available. 
\citet{NCLGL11,NCLGL11b} proposed Bayesian methods to adjust an input 
distribution model to limited data, with a given confidence level. 
When intervals of input variables are given as input information, 
\citet{ZRMM11} and \citet{ZM17} use a family of Johnson distributions to 
represent the uncertainty. 
\citet{CCGLLG16} and \citet{MCCGLG18} assume that the input distribution 
types and parameters follow probability distributions. 
The failure probability is therefore a random variable, 
the confidence level of a reliability constraint, i.e., the probability 
that the failure probability is no greater than a target value, is 
specified. 
To reduce computational cost of this method, \citet{JCL19} proposed a 
so-called reliability approach, inspired by the performance measure 
approach \citep{KL16,LCG10}. 
Subsequently, to further reduce computational cost, \citet{WHYWG20} 
proposed to use the second-order reliability method for computation of 
the failure probability. 
\citet{IKK18} assume that each of random variable follows a normal 
distribution with the mean and the variance modeled as random variables, 
and show that RBDO with a confidence level can be converted to a 
conventional form of RBDO by altering the target reliability index value. 
\citet{ZGXLE20} proposed to use the distributional probability box (the 
distributional p-box) \citep{SS17} for RBDO with limited data of 
uncertain variables. 
\citet{Kan19,Kan20} and \citet{JK20} proposed RBDO methods using order 
statistics. 
These methods, based on the order statistics, 
do not make any assumption on statistical information of 
uncertain parameters, and use random samples of uncertain parameters 
directly to guarantee confidence of the target reliability. 

As reviewed above, most of existing studies on RBDO with uncertainty in the 
input distribution \citep{CCGLLG16,IKK18,MCCGLG18,JCL19,WHYWG20} 
consider probabilistic models of input distribution parameters and/or 
distribution types. 
Accordingly, a confidence level evaluates how the satisfaction of 
structural reliability is reliable. 
In contrast, in this paper we consider a possibilistic model of 
input distribution parameters. 
Hence, what this approach guarantees is a level of robustness 
\citep{Bh06} of the satisfaction of structural reliability. 
A possibilistic model might be, in general, less information-sensitive, 
and hence useful when reliable statistical information of input 
distribution parameters is unavailable. 

From another perspective referring to \citet{SS17}, the uncertainty 
model treated in this paper can be viewed as follows. 
Uncertainty in a structural system is often divided into aleatory 
uncertainty and epistemic uncertainty \citep{OHJWF04}. 
Aleatory uncertainty, i.e., natural variability, is reflected by an 
(uncertain) input distribution. 
Epistemic uncertainty, i.e., state-of-knowledge uncertainty, is 
reflected by uncertainty in the input distribution moments. 
Thus, in our model, aleatory uncertainty is probabilistic, while 
epistemic uncertainty is possibilistic. 
In other words, state-of-knowledge uncertainty is represented as an 
uncertainty set of the input distribution moments. 

Throughout the paper, we assume that only design variables possess 
uncertainty, and that variation of a performance requirement can be 
approximated as a linear function of uncertain perturbations of the 
design variables. 
Also, we do not consider an optimization problem with variation of 
structural topology. 
As for an uncertainty model of moments of the input distribution, we consider 
two concrete convex sets. 
We show that the robust reliability constraint, i.e., constraint that 
the structural reliability is no less than a specified value for any 
possible realizations of input distribution moments, can be reduced to a 
system of nonlinear matrix inequalities. 
This reduction essentially follows the idea presented by 
\citet{EOO03} for computing the worst-case value-at-risk in 
financial engineering.\footnote{%
Diverse extensions of the methodology in 
\citet{EOO03} can be found in literature on so-called 
{\em distributionally robust optimization\/} \citep{DY10,GS10,WKS14}. } 
We can deal with nonlinear matrix inequality constraints 
within the framework of {\em nonlinear semidefinite programming\/} 
(nonlinear SDP) \citep{YY15}. 
In this manner, we can convert an RBDO problem under uncertainty in the 
input distribution moments to a deterministic optimization problem. 
It is worth noting that there exist several applications of 
linear and nonlinear SDPs, as well as eigenvalue optimization, 
to robust design optimization of structures 
\citep{BtN97,KT06,GBZG09,GDG11,TNKK11,HTK15,KZ16,THK17,Kan18}.


The paper is organized as follows. 
In section~\ref{sec:fundamental}, we consider the reliability constraint 
when the input distribution is precisely known, and show some 
fundamental properties. 
Section~\ref{sec:robust} presents the main result; we consider 
uncertainty in the expected value vector and the variance-covariance 
matrix of the input distribution, and examine the constraint that, for 
all possible realizations of the input distribution, the failure 
probability is no greater than a specified value. 
Section~\ref{sec:general} discusses some extensions of the obtained result. 
Section~\ref{sec:ex} presents the results of numerical experiments. 
Section~\ref{sec:conclusion} presents some conclusions.


In our notation, ${}^{\top}$ denotes the transpose of a vector or matrix. 
All vectors are column vectors. 
We use $I$ to denote the identity matrix. 
For two matrices $X=(X_{ij}) \in \Re^{m \times n}$ and 
$Y = (Y_{ij}) \in \Re^{m \times n}$, 
we denote by $X \bullet Y$ the inner product of $X$ and $Y$ defined by 
$X \bullet Y = \tr(X^{\top} Y) = \sum_{i=1}^{n}\sum_{j=1}^{n}X_{ij}Y_{ij}$. 
For a vector $\bi{x} = (x_{i}) \in \Re^{n}$, the notation 
$\| \bi{x} \|_{1}$, $\|\bi{x}\|_{2}$, and $\| \bi{x} \|_{\infty}$ 
designate its $\ell_{1}$-, $\ell_{2}$-, and $\ell_{\infty}$-norms, 
respectively, i.e., 
\begin{align*}
  \| \bi{x} \|_{1}
  &= |x_{1}| + |x_{2}| + \dots + |x_{n}| , \\
  \| \bi{x} \|_{2} 
  &= \sqrt{\bi{x}^{\top} \bi{x}} , \\
  \| \bi{x} \|_{\infty} 
  &= \max \{ |x_{1}|,|x_{2}|,\dots,|x_{n}| \} . 
\end{align*}
For a matrix $X =(X_{ij}) \in \Re^{m \times n}$, define matrix norms 
$\| X \|_{1,1}$, $\| X \|_{\rr{F}}$, and $\| X \|_{\infty,\infty}$ by 
\begin{align*}
  \| X \|_{1,1} 
  &= \sum_{i=1}^{n} \sum_{j=1}^{n} |X_{ij}| , \\
  \| X \|_{\rr{F}} 
  &= \sqrt{X \bullet X} , \\
  \| X \|_{\infty,\infty} 
  &= \max\{ |X_{ij}| 
  \mid i=1,\dots,m, \ j=1,\dots,n \} . 
\end{align*}
Let $\SC^{n}$ denote the set of $n \times n$ symmetric matrices. 
We write $Z \succeq 0$ if $Z \in \SC^{n}$ is positive semidefinite. 
Define $\SC_{+}^{n}$ by 
$\SC_{+}^{n} = \{ Z \in \SC^{n} \mid Z \succeq 0 \}$. 
For a positive definite matrix $Z \in \SC^{n}$, the notation $Z^{1/2}$ 
designates its symmetric square root, i.e., $Z^{1/2} \in \SC^{n}$ 
satisfying $Z^{1/2} Z^{1/2} = Z$. 
We use $Z^{-1/2}$ to denote the inverse matrix of $Z^{1/2}$. 
We use $\NC(\bi{\mu}, \varSigma)$ to denote the multivariate normal 
distribution with an expected value vector $\bi{\mu}$ and a 
variance-covariance matrix $\varSigma$. 
For a random variable $x \in \Re$, its expected value and variance are 
denoted by $\rr{E}[x]$ and $\rr{Var}[x] = \rr{E}[(x - \rr{E}[x])^{2}]$, 
respectively.

\section{Reliability constraint with specified moments}
\label{sec:fundamental}

In this section, we assume that the expected value 
vector and the variance-covariance matrix of the probability 
distribution of the design variable vector are precisely known. 
We first recall the reliability constraint, and then derive its 
alternative expression that will be used in section~\ref{sec:robust} to 
address uncertainty in the probability distribution. 

Let $\bi{x} \in \Re^{n}$ denote a design variable vector, where $n$ is 
the number of design variables. 
Assume that performance requirement in a design optimization problem 
is written in the form 
\begin{align}
  g(\bi{x}) \le 0 , 
  \label{eq:performance.constraint}
\end{align}
where $g : \Re^{n} \to \Re$ is differentiable. 
For simplicity, suppose that the design optimization problem has only 
one constraint; the case where more than one constraints exist will be 
discussed in section~\ref{sec:general}. 

Assume that $\bi{x}$ is decomposed additively as 
\begin{align}
  \bi{x} = \tilde{\bi{x}} + \bi{\zeta} , 
\end{align}
where $\bi{\zeta}$ is a random vector and $\tilde{\bi{x}}$ is a constant 
(i.e., non-random) vector. 
Therefore, in a design optimization problem considered in this paper, 
the decision variable to be optimized is $\tilde{\bi{x}}$. 
We use $\bi{\mu} \in \Re^{n}$ and $\varSigma \in \SC^{n}$ to denote the 
expected value vector and the variance-covariance matrix of $\bi{\zeta}$, 
respectively, i.e., 
\begin{align*}
  \bi{\mu} &= \rr{E}[\bi{\zeta}] , \\
  \varSigma 
  &= \rr{E} \bigl[ (\bi{\zeta}-\rr{E}[\bi{\zeta}]) 
  (\bi{\zeta}-\rr{E}[\bi{\zeta}])^{\top} \bigr] . 
\end{align*}
It is worth noting that $\varSigma$ is positive definite. 
Throughout the paper, we assume that, among parameters in a structural 
system, only $\bi{\zeta}$ possesses uncertainty. 
Also, we restrict ourselves to optimization without change of 
structural topology; i.e., we do not consider topology 
optimization.\footnote{%
In topology optimization, it would be proper to consider the design 
variables of removed structural elements as non-random variables. 
In this paper we do not discuss this issue. }

For simplicity and clarity of discussion, we assume 
$\bi{\zeta} \sim \NC(\bi{\mu}, \varSigma)$ in 
section~\ref{sec:fundamental} and section~\ref{sec:robust}. 
In fact, the results established in these sections can be extended to 
the case that the type of probability distribution is unknown; 
we then require that the reliability constraint should be satisfied for 
{\em any\/} probability distribution with moments belonging to a 
specified set. 
We defer this case until section~\ref{sec:general}. 

Since $\bi{x}$ is a random vector, $g(\bi{x})$ is a random variable. 
Therefore, constraint \eqref{eq:performance.constraint} should be 
considered in a probabilistic sense, which yields the reliability 
constraint 
\begin{align}
  \rr{P}_{\NC(\bi{\mu},\varSigma)} \{ g(\bi{x}) \le 0 \} 
  \ge 1 - \epsilon . 
  \label{eq:reliability.constraint.1}
\end{align}
Here, $\epsilon \in ]0,1]$ is the specified upper bound for the failure 
probability. 
Let $g^{\rr{lin}}(\bi{x})$ denote the first-order approximation of 
$g(\bi{x})$ centered at $\bi{x}=\tilde{\bi{x}}$, i.e., 
\begin{align*}
  g^{\rr{lin}}(\bi{x}) 
  = g(\tilde{\bi{x}}) + \nabla g(\tilde{\bi{x}})^{\top} \bi{\zeta} 
  \,(\simeq  g(\bi{x}) ) . 
\end{align*}
Throughout the paper, we consider an approximation of constraint 
\eqref{eq:reliability.constraint.1} 
\begin{align}
  \rr{P}_{\NC(\bi{\mu},\varSigma)} \{
  g^{\rr{lin}}(\bi{x}) \le 0  
  \} 
  \ge 1 - \epsilon , 
  \label{eq:reliability.constraint.2.0}
\end{align}
i.e., 
\begin{align}
  \rr{P}_{\NC(\bi{\mu},\varSigma)} \{
  g(\tilde{\bi{x}}) + \nabla g(\tilde{\bi{x}})^{\top} \bi{\zeta}  \le 0  
  \} 
  \ge 1 - \epsilon . 
  \label{eq:reliability.constraint.2}
\end{align}
Therefore, the corresponding RBDO problem has the following form: 
\begin{subequations}\label{P.reliability.1}%
  \begin{alignat}{3}
    & \MIN
    &{\quad}& 
    f(\bi{x}) \\
    & \ST && 
    \bi{x} \in X , 
    \label{P.reliability.1.3} \\
    & && 
    \rr{P}_{\NC(\bi{\mu},\varSigma)} \{
    g(\tilde{\bi{x}}) + \nabla g(\tilde{\bi{x}})^{\top} \bi{\zeta}  \le 0  
    \}
    \ge 1 - \epsilon . 
  \end{alignat}
\end{subequations}
Here, $f : \Re^{n} \to \Re$ is the objective function, 
$X \subseteq \Re^{n}$ is a given closed set, and constraint $\bi{x} \in X$ 
corresponds to, e.g., the side constraints on the design variables.

From the basic property of the normal distribution, we can readily 
obtain the following reformulation of the reliability constraint. 
\begin{theorem}\label{thm:reliability.constraint}
  Define $\kappa$ by 
  \begin{align*}
    \kappa = - \varPhi^{-1}(\epsilon) , 
  \end{align*}
  where $\varPhi$ is the (cumulative) distribution function of the 
  standard normal distribution $\NC(0,1)$. 
  Then, $\tilde{\bi{x}} \in X$ satisfies 
  \eqref{eq:reliability.constraint.2} if and only if it satisfies 
  \begin{align}
    g(\tilde{\bi{x}}) 
    + \nabla g(\tilde{\bi{x}})^{\top} \bi{\mu} 
    + \kappa 
    \| \varSigma^{1/2} \nabla g(\tilde{\bi{x}}) \|_{2} \le 0 . 
    \label{eq:reliability.constraint.3}
  \end{align}
\end{theorem}
\begin{proof}
  Since $g^{\rr{lin}}(\bi{x})$ follows the normal distribution, 
  it is standardized by 
  \begin{align*}
    z 
    = \frac{g^{\rr{lin}}(\bi{x}) - \rr{E}[g^{\rr{lin}}(\bi{x})]}
    {\sqrt{\rr{Var}[g^{\rr{lin}}(\bi{x})]}}
    \sim \NC(0,1) . 
  \end{align*}
  By using this relation, we can eliminate 
  $g^{\rr{lin}}(\bi{x})$ from \eqref{eq:reliability.constraint.2.0} 
  (i.e., \eqref{eq:reliability.constraint.2}) as 
  \begin{align*}
    \rr{P}_{\NC(0,1)} \left\{
    z \le 
    - \frac{\rr{E}[g^{\rr{lin}}(\bi{x})]}{\sqrt{\rr{Var}[g^{\rr{lin}}(\bi{x})]}}
    \right\} 
    \ge 1 - \epsilon . 
  \end{align*}
  This inequality is equivalently rewritten by using the distribution 
  function $\varPhi$ as 
  \begin{align*}
    - \frac{\rr{E}[g^{\rr{lin}}(\bi{x})]}{\sqrt{\rr{Var}[g^{\rr{lin}}(\bi{x})]}}
    \ge \varPhi^{-1}(1-\epsilon) = -\varPhi^{-1}(\epsilon) . 
  \end{align*}
  By direct calculations, we see that the expected value of 
  $g^{\rr{lin}}(\bi{x})$ is 
  \begin{align*}
    \rr{E}[g^{\rr{lin}}(\bi{x})] 
    = g(\tilde{\bi{x}}) 
    + \nabla g(\tilde{\bi{x}})^{\top} \bi{\mu}  
  \end{align*}
  and the variance is 
  \begin{align*}
    \rr{Var}[g^{\rr{lin}}(\bi{x})] 
    &= \rr{E}
    \bigl[ (g^{\rr{lin}}(\bi{x}) - \rr{E}[g^{\rr{lin}}(\bi{x})])^{2} \bigr] 
    \notag\\
    &= \rr{E}
    \bigl[ \bigl(
    (g(\tilde{\bi{x}}) + \nabla g(\tilde{\bi{x}})^{\top} \bi{\zeta})
    - (g(\tilde{\bi{x}}) + \nabla g(\tilde{\bi{x}})^{\top} \bi{\mu})
    \bigr)^{2} \bigr] 
    \notag\\
    &= \rr{E} \bigl[ 
    \bigl(
    \nabla g(\tilde{\bi{x}})^{\top} (\bi{\zeta} - \bi{\mu})
    \bigr)^{2}
    \bigr] 
    \notag\\
    &= \rr{E} \bigl[ 
    \nabla g(\tilde{\bi{x}})^{\top} (\bi{\zeta} - \bi{\mu}) 
    (\bi{\zeta} - \bi{\mu})^{\top} \nabla g(\tilde{\bi{x}}) 
    \bigr] 
    \notag\\
    &= \nabla g(\tilde{\bi{x}})^{\top} \rr{E} \bigl[ 
    (\bi{\zeta} - \bi{\mu}) 
    (\bi{\zeta} - \bi{\mu})^{\top}
    \bigr]  \nabla g(\tilde{\bi{x}}) 
    \notag\\
    &= \nabla g(\tilde{\bi{x}})^{\top} \varSigma 
    \nabla g(\tilde{\bi{x}}) , 
  \end{align*}
  which concludes the proof. 
\end{proof}

In section~\ref{sec:robust}, we deal with the case in which $\bi{\mu}$  
and $\varSigma$ are known imprecisely. 
To do this, we reformulate 
$\kappa \| \varSigma^{1/2} \nabla g(\tilde{\bi{x}}) \|_{2}$ in 
\eqref{eq:reliability.constraint.3} into a form suitable for analysis. 
The following theorem is obtained in the same manner as 
\citet[Theorem~1]{EOO03}. 

\begin{theorem}\label{thm:SDP.duality}
  For $\kappa > 0$, $\varSigma \in \SC_{+}^{n}$, and 
  $\nabla g(\tilde{\bi{x}}) \in \Re^{n}$, we have 
  \begin{align*}
    \kappa \| \varSigma^{1/2} \nabla g(\tilde{\bi{x}}) \|_{2} 
    = \min_{\varLambda \in \SC^{n}, \, z \in \Re}
    \left\{
    \varSigma \bullet \varLambda 
    + \kappa^{2} z 
    \left|
    \begin{bmatrix}
      \varLambda & \nabla g(\tilde{\bi{x}})/2 \\
      \nabla g(\tilde{\bi{x}})^{\top}/2 & z \\
    \end{bmatrix}
    \succeq 0 
    \right.
    \right\} . 
  \end{align*}
\end{theorem}
\begin{proof}
  We first show that the left side of the equation can be reduced to 
  \begin{align}
    \kappa  \| \varSigma^{1/2} \nabla g(\tilde{\bi{x}}) \|_{2} 
    &= \max_{\bi{\zeta} \in \Re^{n}} 
    \{ \nabla g(\tilde{\bi{x}})^{\top} \bi{\zeta} 
    \mid
    \| \varSigma^{-1/2} \bi{\zeta} \|_{2} = \kappa 
    \} . 
    \label{eq:evaluate.dual.1}
  \end{align}
  To see this, we apply the Lagrange multiplier method to the equality 
  constrained maximization problem on the right side of 
  \eqref{eq:evaluate.dual.1}. 
  Namely, the Lagrangian $L_{1}: \Re^{n} \times \Re \to \Re$ is defined by 
  \begin{align*}
    L_{1}(\bi{\zeta}; \mu) 
    = \nabla g(\tilde{\bi{x}})^{\top} \bi{\zeta} 
    + \frac{\mu}{2} 
    (\kappa^{2} - \bi{\zeta}^{\top} \varSigma^{-1} \bi{\zeta} ) , 
  \end{align*}
  where $\mu \in \Re$ is the Lagrange multiplier. 
  The stationarity condition of $L_{1}$ is 
  \begin{align*}
    \pdif{L_{1}}{\bi{\zeta}} 
    &= \nabla g(\tilde{\bi{x}}) - \mu \varSigma^{-1} \bi{\zeta} = \bi{0} , \\
    \pdif{L_{1}}{\mu} 
    &= \kappa^{2} - \bi{\zeta}^{\top} \varSigma^{-1} \bi{\zeta} = 0 . 
  \end{align*}
  By solving this stationarity condition, we can find that 
  \begin{align*}
    \bi{\zeta} = \dfrac{1}{\mu} \varSigma \nabla g(\tilde{\bi{x}}) , 
    \quad
    \mu 
    = \frac{\| \varSigma^{1/2} \nabla g(\tilde{\bi{x}}) \|_{2}}{\kappa}
  \end{align*}
  are optimal. Hence, the optimal value is 
  \begin{align*}
    \nabla g(\tilde{\bi{x}})^{\top} \bi{\zeta} 
    = \nabla g(\tilde{\bi{x}})^{\top} 
    \frac{\kappa \varSigma \nabla g(\tilde{\bi{x}})}
    {\| \varSigma^{1/2} \nabla g(\tilde{\bi{x}}) \|_{2}} , 
  \end{align*}
  which is reduced to the left side of \eqref{eq:evaluate.dual.1}. 
  
  Next, observe that the right side of \eqref{eq:evaluate.dual.1} is 
  further reduced to 
  \begin{align}
    \MoveEqLeft
    \max_{\bi{\zeta} \in \Re^{n}} 
    \{ \nabla g(\tilde{\bi{x}})^{\top} \bi{\zeta} 
    \mid
    \| \varSigma^{-1/2} \bi{\zeta} \|_{2} = \kappa  \}   \notag\\
    &= \max_{\bi{\zeta} \in \Re^{n}} 
    \{ \nabla g(\tilde{\bi{x}})^{\top} \bi{\zeta} 
    \mid
    \| \varSigma^{-1/2} \bi{\zeta} \|_{2} \le \kappa 
    \}  \notag\\
    &= \max_{\bi{\zeta} \in \Re^{n}} 
    \left\{ \nabla g(\tilde{\bi{x}})^{\top} \bi{\zeta} 
    \left|
    \begin{bmatrix}
      \varSigma & \bi{\zeta} \\
      \bi{\zeta}^{\top} & \kappa^{2}
    \end{bmatrix}
    \succeq 0 
    \right. \right\} . 
    \label{eq:evaluate.dual.2}
  \end{align}
  Here, the last equality follows from the  fact that the positive 
  semidefinite constraint is equivalent to the nonnegative constraint on 
  the Schur complement of $\varSigma$ in the corresponding matrix, i.e., 
  $\kappa^{2} - \bi{\zeta}^{\top} \varSigma^{-1} \bi{\zeta} \ge 0$; 
  see \cite[appendix~A.5.5]{BV04}. 
  It is worth noting that the last expression in 
  \eqref{eq:evaluate.dual.2} is an SDP problem. 
  
  Finally, we shall show that the right side of the proposition in this 
  theorem corresponds to the dual problem of the SDP 
  problem in \eqref{eq:evaluate.dual.2}. 
  Since this dual problem is strictly feasible, 
  the proposition follows from the strong 
  duality of SDP \citep[section~11.3]{CEg14}. 
  We can derive the dual problem of \eqref{eq:evaluate.dual.2} as follows. 
  The Lagrangian is defined by 
  \begin{align}
    L_{2}(\bi{\zeta}; \varLambda, \bi{\lambda}, z) = 
    \begin{dcases*}
      \nabla g(\tilde{\bi{x}})^{\top} \bi{\zeta} + 
      \begin{bmatrix}
        \varLambda & \bi{\lambda} \\
        \bi{\lambda}^{\top} & z \\
      \end{bmatrix}
      \bullet
      \begin{bmatrix}
        \varSigma & \bi{\zeta} \\
        \bi{\zeta}^{\top} & \kappa^{2}
      \end{bmatrix}
      & if 
      $\begin{bmatrix}
        \varLambda & \bi{\lambda} \\
        \bi{\lambda}^{\top} & z \\
      \end{bmatrix} \succeq 0$, \\
      +\infty
      & otherwise, 
    \end{dcases*}
    \label{eq:SDP.duality.Lagrangian}
  \end{align}
  where $z \in \Re$, $\bi{\lambda} \in \Re^{n}$, 
  and $\varLambda \in \SC^{n}$ are the Lagrange multipliers. 
  Indeed, since the positive semidefinite cone 
  satisfies \citep[Fact~1.3.17]{Kan11}
  \begin{align}
    \inf_{S \in \SC^{n}} 
    \{ S \bullet T \mid S \succeq 0 \} = 
    \begin{dcases*}
      0 & if $T \succeq 0$, \\
      -\infty & otherwise, 
    \end{dcases*}
    \label{eq:PSD.self-duality}
  \end{align}
  we can confirm that the SDP problem in \eqref{eq:evaluate.dual.2} 
  is equivalent to 
  \begin{align*}
    \max_{\bi{\zeta}} 
    \inf_{\varLambda,\, \bi{\lambda},\, z} 
     L_{2}(\bi{\zeta}; \varLambda, \bi{\lambda}, z) . 
  \end{align*}
  The dual problem is defined by 
  \begin{align}
    \min_{\varLambda,\, \bi{\lambda},\, z} 
    \sup_{\bi{\zeta}} 
     L_{2}(\bi{\zeta}; \varLambda, \bi{\lambda}, z) . 
    \label{eq:SDP.duality.dual.problem}
  \end{align}
  Since \eqref{eq:SDP.duality.Lagrangian} can be rewritten as 
  \begin{align*}
    L_{2}(\bi{\zeta}; \varLambda, \bi{\lambda}, z) = 
    \begin{dcases*}
      (\nabla g(\tilde{\bi{x}}) + 2\bi{\lambda})^{\top} \bi{\zeta} 
      + \varSigma \bullet \varLambda  + \kappa^{2} z
      & if 
      $\begin{bmatrix}
        \varLambda & \bi{\lambda} \\
        \bi{\lambda}^{\top} & z \\
      \end{bmatrix} \succeq 0$, \\
      +\infty
      & otherwise, 
    \end{dcases*}
  \end{align*}
  we have 
  \begin{align*}
    \sup_{\bi{\zeta}} 
    L_{2}(\bi{\zeta}; \varLambda, \bi{\lambda}, z) = 
    \begin{dcases*}
      \varSigma \bullet \varLambda  + \kappa^{2} z
      & if 
      $\begin{bmatrix}
        \varLambda & \bi{\lambda} \\
        \bi{\lambda}^{\top} & z \\
      \end{bmatrix} \succeq 0$, 
      $\nabla g(\tilde{\bi{x}}) + 2\bi{\lambda} = \bi{0}$, \\
      +\infty
      & otherwise. 
    \end{dcases*}
  \end{align*}
  Therefore, the dual problem in \eqref{eq:SDP.duality.dual.problem} 
  corresponds to the right side of the proposition of the theorem. 
\end{proof}

\section{Worst-case reliability under uncertainty in moments}
\label{sec:robust}

In this section, we consider the case that the moments (in this paper, the 
expected value vector and the variance-covariance matrix) of the design 
variable vector are uncertain, or not perfectly known. 
Specifically, they are only known to be in a given set, called 
the {\em uncertainty set\/}. 
We require that a structure satisfies the reliability constraint for 
{\em any\/} moments in the uncertainty set. 
In other words, we require that the failure probability in the worst 
case is not larger than a specified value. 
We show that this requirement can be converted to a form of conventional 
constraints in deterministic optimization. 

\subsection{Convex uncertainty model of moments}
\label{sec:robust.convex}

Let $U_{\bi{\mu}} \subset \Re^{n}$ and 
$U_{\varSigma} \subset \SC_{+}^{n}$ denote the uncertainty sets, i.e., 
the sets of all possible realizations, of $\bi{\mu}$ and $\varSigma$, 
respectively. 
Namely, we only know that $\bi{\mu}$ and $\varSigma$ satisfy 
\begin{align*}
  \bi{\mu}  &\in U_{\bi{\mu}} ,  \\
  \varSigma &\in U_{\varSigma} . 
\end{align*}
Assume that $U_{\bi{\mu}}$ and $U_{\varSigma}$ are compact convex sets. 
For notational simplicity, we write $(\bi{\mu},\varSigma) \in U$ 
if $\bi{\mu} \in U_{\bi{\mu}}$ and $\varSigma \in U_{\varSigma}$ hold. 

Recall that we are considering the reliability constraint in 
\eqref{eq:reliability.constraint.2} with a linearly approximated 
constraint function. 
The robust counterpart of \eqref{eq:reliability.constraint.2} against 
uncertainty in $\bi{\mu}$ and $\varSigma$ is formulated as 
\begin{align}
  \rr{P}_{\NC(\bi{\mu},\varSigma)} 
  \{ g(\tilde{\bi{x}}) + \nabla g(\tilde{\bi{x}})^{\top} \bi{\zeta}  \le 0 \} 
  \ge 1 - \epsilon  , 
  \quad \forall (\bi{\mu},\varSigma) \in U  , 
  \label{eq:robust.reliable.infinite}
\end{align}
i.e., we require that the reliability constraint should be satisfied for 
any normal distribution corresponding to possible realizations of 
$\bi{\mu}$ and $\varSigma$. 
This requirement is equivalently rewritten as 
\begin{align}
  \max_{(\bi{\mu},\varSigma) \in U} \bigl\{
  \rr{P}_{\NC(\bi{\mu},\varSigma)} 
  \{ g(\tilde{\bi{x}}) + \nabla g(\tilde{\bi{x}})^{\top} \bi{\zeta}  \le 0 \} 
  \bigr\}  \ge 1 - \epsilon  . 
  \label{eq:robust.reliable}
\end{align}
That is, the reliability constraint should be satisfied in the worst 
case. 

The following theorem presents, with the aid of 
\refthm{thm:reliability.constraint} and \refthm{thm:SDP.duality}, an 
equivalent reformulation of \eqref{eq:robust.reliable}. 

\begin{theorem}\label{thm:robust.general}
  $\tilde{\bi{x}} \in X$ satisfies \eqref{eq:robust.reliable} if and 
  only if there exists a pair of $z \in \Re$ and 
  $\varLambda \in \SC^{n}$ satisfying 
  \begin{align}
    & g(\tilde{\bi{x}})  
    + \max \{ \nabla g(\tilde{\bi{x}})^{\top} \bi{\mu} 
    \mid \bi{\mu} \in U_{\bi{\mu}} \} 
    + \max \{ \varSigma \bullet \varLambda  
    \mid \varSigma \in U_{\varSigma} \} 
    + \kappa^{2} z \le 0 , 
    \label{eq:robust.general.1} \\
    & 
    \begin{bmatrix}
      \varLambda & \nabla g(\tilde{\bi{x}})/2 \\
      \nabla g(\tilde{\bi{x}})^{\top}/2 & z \\
    \end{bmatrix}
    \succeq 0 . 
    \label{eq:robust.general.2}
  \end{align}
\end{theorem}
\begin{proof}
  It follows from \refthm{thm:reliability.constraint} that 
  \eqref{eq:robust.reliable} is equivalent to 
  \begin{align}
    \max_{(\bi{\mu},\varSigma) \in U} 
    \{   g(\tilde{\bi{x}}) 
    + \nabla g(\tilde{\bi{x}})^{\top} \bi{\mu} 
    + \kappa 
    \| \varSigma^{1/2} \nabla g(\tilde{\bi{x}}) \| \} \le 0 .  
    \label{eq:robust.constraint.3}
  \end{align}
  Furthermore, application of \refthm{thm:SDP.duality} yields 
  \begin{align}
    \max_{(\bi{\mu},\varSigma) \in U} 
    \left\{
    g(\tilde{\bi{x}}) 
    + \nabla g(\tilde{\bi{x}})^{\top} \bi{\mu} 
    + \min_{z, \, \varLambda}
    \left\{
    \varSigma \bullet \varLambda 
    + \kappa^{2} z 
    \left|
    \begin{bmatrix}
      \varLambda & \nabla g(\tilde{\bi{x}})/2 \\
      \nabla g(\tilde{\bi{x}})^{\top}/2 & z \\
    \end{bmatrix}
    \succeq 0 
    \right.
    \right\}
    \right\}
    \le 0 . 
    \label{eq:robust.general.3}
  \end{align}
  In the expression above, we see that $U$ is compact and convex, and the 
  feasible set for the minimization is convex. 
  Also, the objective function is linear in $\bi{\mu}$ and $\varSigma$ 
  for fixed $z$ and $\varLambda$, and is linear in $z$ and $\varLambda$ 
  for fixed $\bi{\mu}$ and $\varSigma$. 
  Therefore, the minimax theorem \citep[Theorem~8.8]{CEg14} asserts that 
  \eqref{eq:robust.general.3} is equivalent to 
  \begin{align*}
    g(\tilde{\bi{x}})  + 
    \min_{z, \, \varLambda} 
    \max_{(\bi{\mu},\varSigma) \in U} 
    \left\{
    \nabla g(\tilde{\bi{x}})^{\top} \bi{\mu} 
    + \varSigma \bullet \varLambda 
    + \kappa^{2} z 
    \left|
    \begin{bmatrix}
      \varLambda & \nabla g(\tilde{\bi{x}})/2 \\
      \nabla g(\tilde{\bi{x}})^{\top}/2 & z \\
    \end{bmatrix}
    \succeq 0 
    \right.
    \right\}
    \le 0 . 
  \end{align*}
  This inequality holds if and only if there exists a feasible pair of 
  $z\in \Re$ and $\varLambda \in \SC^{n}$ satisfying 
  \begin{align*}
    g(\tilde{\bi{x}})  + 
    \max_{(\bi{\mu},\varSigma) \in U} 
    \{
    \nabla g(\tilde{\bi{x}})^{\top} \bi{\mu} 
    + \varSigma \bullet \varLambda 
    + \kappa^{2} z 
    \}
    \le 0 , 
  \end{align*}
  which concludes the proof. 
\end{proof}

The conclusion of \refthm{thm:robust.general} is quite abstract in the 
sense that concrete forms of $U_{\bi{\mu}}$ and $U_{\varSigma}$ are not 
specified. 
To use this result into design optimization in practice, we have 
to reduce 
$\max \{ \nabla g(\tilde{\bi{x}})^{\top} \bi{\mu} 
  \mid \bi{\mu} \in U_{\bi{\mu}} \}$ and 
$\max \{ \varSigma \bullet \varLambda  
  \mid \varSigma \in U_{\varSigma} \}$ 
in \eqref{eq:robust.general.1} to tractable forms. 
This is actually performed in section~\ref{sec:robust.ell.infty} and 
section~\ref{sec:robust.ell.2}, where we consider two specific models of 
$U_{\bi{\mu}}$ and $U_{\varSigma}$.

\subsection{Uncertainty model with $\ell_{\infty}$-norm}
\label{sec:robust.ell.infty}

Let $\tilde{\bi{\mu}} \in \Re^{n}$ and 
$\tilde{\varSigma} \in \SC^{n}$ denote the best estimates of 
$\bi{\mu}$ and $\varSigma$, respectively, 
where $\tilde{\varSigma}$ is positive definite. 
In this section, we specialize the results of 
section~\ref{sec:robust.convex} to the case that the uncertainty sets 
are given as 
\begin{align}
  U_{\bi{\mu}} 
  & = \{ \tilde{\bi{\mu}} + A \bi{z}_{1} 
  \mid
  \| \bi{z}_{1} \|_{\infty} \le \alpha , \
  \bi{z}_{1} \in \Re^{m} \} , 
  \label{eq:uncertainty.set.ell.infty.1} \\
  U _{\varSigma} 
  &= \{ \tilde{\varSigma} + B Z_{2} B^{\top} 
  \mid 
  \| Z_{2} \|_{\infty,\infty} \le \beta , \
  Z_{2} \in \SC^{k} 
  \} 
  \cap \SC_{+}^{n} . 
  \label{eq:uncertainty.set.ell.infty.2}
\end{align}
Here, $\bi{z}_{1} \in \Re^{m}$ and $Z_{2} \in \SC^{k}$ are unknown 
vector and matrix reflecting the uncertainty in $\bi{\mu}$ and 
$\varSigma$, respectively, 
$A \in \Re^{n \times m}$ and $B \in \Re^{n \times k}$ are constant 
matrices, 
and $\alpha$ and $\beta$ are nonnegative parameters representing the 
magnitude of uncertainties. 

\begin{example}\label{ex:ell.infty.uncertainty}
  A simple example of the uncertainty set in 
  \eqref{eq:uncertainty.set.ell.infty.1} is a box-constrained model. 
  For example, if we put $\tilde{\bi{\mu}}=\bi{0}$ and $A=I$ with $m=n$, 
  \eqref{eq:uncertainty.set.ell.infty.1} is reduced to 
  \begin{align*}
    U_{\bi{\mu}} 
    = \{ \bi{z}_{1} \in \Re^{n} 
    \mid
    \| \bi{z}_{1} \|_{\infty} \le \alpha  \} . 
  \end{align*}
  This means that the expected value vector $\bi{\mu}$ belongs to a 
  hypercube centered at the origin, with edges parallel to the axes and 
  with an edge length of $2\alpha$. 
  In other words, each component $\mu_{j}$ of $\bi{\mu}$ can take any 
  value in $[-\alpha,\alpha]$. 
  Similarly, a simple example of the uncertainty set in 
  \eqref{eq:uncertainty.set.ell.infty.2} is the one with 
  $B=I$ and $k=n$, i.e., 
  \begin{align*}
    U_{\varSigma} 
    = \{ \tilde{\varSigma} + Z_{2} 
    \mid 
    \| Z_{2} \|_{\infty,\infty} \le \beta , \
    Z_{2}^{\top} = Z_{2} 
    \} 
    \cap \SC_{+}^{n} . 
  \end{align*}
  This means that, roughly speaking, the variance-covariance matrix 
  $\varSigma$ has componentwise uncertainty. 
  More precisely, for each $i$, $j=1,\dots,n$ we have 
  \begin{align}
    \tilde{\varSigma}_{ij} - \beta 
    \le \varSigma_{ij} \le \tilde{\varSigma}_{ij} + \beta , 
    \quad
    \varSigma_{ji} = \varSigma_{ij} , 
    \label{eq:ell1.uncertainty.example.1}
  \end{align}
  and besides $\varSigma$ should be positive semidefinite. 
  It is worth noting that, even if $\tilde{\varSigma}$ and $\beta$ 
  satisfy $\tilde{\varSigma} - \beta\bi{1}\bi{1}^{\top} \succ 0$ and 
  $\tilde{\varSigma} + \beta\bi{1}\bi{1}^{\top} \succ 0$ 
  (here, $\bi{1}$ denotes an all-ones column vector), 
  \eqref{eq:ell1.uncertainty.example.1} does not necessarily imply 
  $\varSigma \succ 0$. 
  Indeed, as for an example with $n=2$, consider 
  \begin{align*}
    \tilde{\varSigma} = 
    \begin{bmatrix}
      3 & 2 \\ 2 & 3 \\
    \end{bmatrix}
    , \quad
    \beta = 2 . 
  \end{align*}
  Then we have 
  \begin{align*}
    \tilde{\varSigma} - \beta\bi{1}\bi{1}^{\top} =
    \begin{bmatrix}
      1 & 0 \\ 0 & 1 \\
    \end{bmatrix}
    \succ 0 
    , \quad
    \tilde{\varSigma} + \beta\bi{1}\bi{1}^{\top} = 
    \begin{bmatrix}
      5 & 4 \\ 4 & 5 \\
    \end{bmatrix}
    \succ 0 , 
  \end{align*}
  and, for example, we see that 
  \begin{align*}
    \varSigma = 
    \begin{bmatrix}
      2 & 3 \\ 3 & 2 \\
    \end{bmatrix}
  \end{align*}
  satisfies \eqref{eq:ell1.uncertainty.example.1} but 
  $\varSigma \not\succ 0$. 
  \finbox
\end{example}

To derive the main result in this section stated in 
\refthm{thm:ell.infty.norm}, we need the two technical lemmas. 
\reflmm{lmm:ell1.norm.mu} explicitly computes the value of 
$\max  \{ \nabla g(\tilde{\bi{x}})^{\top} \bi{\mu} 
  \mid \bi{\mu} \in U_{\bi{\mu}} \}$ in \eqref{eq:robust.general.1}. 
\reflmm{lmm:ell1.norm.varSigma} converts 
$\max \{ \varSigma \bullet \varLambda 
  \mid \varSigma \in U_{\varSigma}\}$ 
to a tractable form. 

\begin{lemma}\label{lmm:ell1.norm.mu}
  For $U_{\bi{\mu}}$ defined by \eqref{eq:uncertainty.set.ell.infty.1}, 
  we have 
  \begin{align*}
    \max_{\bi{\mu} \in U_{\bi{\mu}}} \{ 
    \nabla g(\tilde{\bi{x}})^{\top} \bi{\mu} 
    \} 
    = \nabla g(\tilde{\bi{x}})^{\top} \tilde{\bi{\mu}} 
    + \alpha \| A^{\top} \nabla g(\tilde{\bi{x}}) \|_{1}  . 
  \end{align*}
\end{lemma}
\begin{proof}
  Substitution of \eqref{eq:uncertainty.set.ell.infty.1} into the left 
  side yields 
  \begin{align*}
    \MoveEqLeft
    \max_{\bi{\mu} \in U_{\bi{\mu}}} \{ 
    \nabla g(\tilde{\bi{x}})^{\top} \bi{\mu} 
    \}   \notag\\
    &= \nabla g(\tilde{\bi{x}})^{\top} \tilde{\bi{\mu}} 
    + \max_{\bi{z}_{1}} \{ 
    \nabla g(\tilde{\bi{x}})^{\top} A \bi{z}_{1} 
    \mid
    \| \bi{z}_{1} \|_{\infty} \le \alpha 
    \}   \notag\\
    &= \nabla g(\tilde{\bi{x}})^{\top} \tilde{\bi{\mu}} 
    + \alpha \max_{\bi{z}_{1}} \{ 
    ( A^{\top} \nabla g(\tilde{\bi{x}}) )^{\top} \bi{z}_{1} 
    \mid
    \| \bi{z}_{1} \|_{\infty} \le 1 
    \} . 
  \end{align*}
  It is known that the dual norm of the $\ell_{\infty}$-norm is 
  the $\ell_{1}$-norm \cite[appendix~A.1.6]{BV04}, i.e., 
  \begin{align*}
    \max_{\bi{t} \in \Re^{n}} \{ \bi{s}^{\top} \bi{t} 
    \mid \| \bi{t} \|_{\infty} \le 1 \}  
    = \| \bi{s} \|_{1} . 
  \end{align*}
  Therefore, we obtain 
  \begin{align*}
    \max_{\bi{z}_{1}} \{ 
    ( A^{\top} \nabla g(\tilde{\bi{x}}) )^{\top} \bi{z}_{1} 
    \mid
    \| \bi{z}_{1} \|_{\infty} \le 1
    \}
    = \| A^{\top} \nabla g(\tilde{\bi{x}}) \|_{1} , 
  \end{align*}
  which concludes the proof. 
\end{proof}

\begin{lemma}\label{lmm:ell1.norm.varSigma}
  For $U _{\varSigma}$ defined by \eqref{eq:uncertainty.set.ell.infty.2}, 
  we have 
  \begin{align*}
    \max_{\varSigma \in U_{\varSigma}} \{
    \varLambda \bullet \varSigma  \} 
    = \min_{\varOmega \in \SC_{+}^{k}} \{
    \tilde{\varSigma} \bullet (\varLambda + \varOmega) 
    + \beta \| B^{\top} (\varLambda + \varOmega) B \|_{1,1}  \} . 
  \end{align*}
\end{lemma}
\begin{proof}
  We shall show that the right side corresponds to the dual problem of 
  the SDP problem on the left side. 
  Therefore, this proposition follows from the strong duality of 
  SDP \citep[section~11.3]{CEg14}, because the dual problem is strictly 
  feasible. 
  
  As preliminaries, for a convex cone defined by 
  $K = \{ (s,S) \in \Re \times \SC^{k} \mid \| S \|_{1,1} \le s \}$, 
  observe that its dual cone is given by \cite[Example~2.25]{BV04} 
  \begin{align*}
    \MoveEqLeft
    \{ (t,T) \in \Re \times \SC^{k}  
    \mid s t + S \bullet T \ge 0 \ (\forall (s,S) \in K) \}   \notag\\
    &= \{ (t,T) \in \Re \times \SC^{k} \mid \| T \|_{\infty,\infty} \le t \} ,
  \end{align*}
  from which we obtain 
  \begin{align}
    \inf_{s \in \Re, \, S \in \SC^{k}} 
    \{ s t + S \bullet T  \mid \| S \|_{1,1} \le s \} = 
    \begin{dcases*}
      0 
      & if $\| T \|_{\infty,\infty} \le t$, \\
      -\infty 
      & otherwise. 
    \end{dcases*}
    \label{eq:ell1.dual.norm.1}
  \end{align}
  
  By using definition \eqref{eq:uncertainty.set.ell.infty.2} of 
  $U _{\varSigma}$, the left side of the proposition of this theorem is 
  reduced to 
  \begin{align}
    \max_{\varSigma \in U_{\varSigma}} \{
    \varLambda \bullet \varSigma  \} 
    = \max_{Z_{2} \in \SC^{k}} \{
    \varLambda \bullet (\tilde{\varSigma} + B Z_{2} B^{\top}) 
    \mid
    \| Z_{2} \|_{\infty,\infty} \le \beta , \
    \tilde{\varSigma} + B Z_{2} B^{\top} \succeq 0 \} . 
    \label{eq:ell1.dual.norm.2}
  \end{align}
  The Lagrangian of this optimization problem is defined by 
  \begin{align}
    L(Z_{2}; v, V, \varOmega) = 
    \begin{dcases*}
      \varLambda \bullet (\tilde{\varSigma} + B Z_{2} B^{\top}) 
      + (\beta v + Z_{2} \bullet V ) \\
      \qquad{}+ 
      \varOmega \bullet (\tilde{\varSigma} + B Z_{2} B^{\top}) 
      & if $\| V \|_{1,1} \le v$, $\varOmega \succeq 0$, \\
      +\infty 
      & otherwise, 
    \end{dcases*}
    \label{eq:ell1.dual.Lagrangian.1}
  \end{align}
  where $v \in \Re$, $V \in \SC^{k}$, and $\varOmega \in \SC^{n}$ are 
  the Lagrange multipliers. 
  Indeed, by using \eqref{eq:PSD.self-duality} and 
  \eqref{eq:ell1.dual.norm.1}, we can confirm that problem 
  \eqref{eq:ell1.dual.norm.2} is equivalent to 
  \begin{align*}
    \max_{Z_{2}} 
    \inf_{v, \, V, \, \varOmega} L(Z_{2}; v, V, \varOmega) . 
  \end{align*}
  The dual problem is then defined by 
  \begin{align}
    \min_{v, \, V, \, \varOmega} 
    \sup_{Z_{2}} L(Z_{2}; v, V, \varOmega) . 
    \label{eq:ell1.dual.problem}
  \end{align}
  Since \eqref{eq:ell1.dual.Lagrangian.1} can be rewritten as 
  \begin{align*}
    L(Z_{2}; v, V, \varOmega) = 
    \begin{dcases*}
      Z_{2} \bullet (V + B^{\top} (\varLambda + \varOmega) B) \\
      \qquad{}+ \tilde{\varSigma} \bullet (\varLambda + \varOmega) 
      + \beta v  
      & if $\| V \|_{1,1} \le v$, $\varOmega \succeq 0$, \\
      +\infty 
      & otherwise, 
    \end{dcases*}
  \end{align*}
  we obtain 
  \begin{align*}
    \sup_{Z_{2}}
    L(Z_{2}; v, V, \varOmega) = 
    \begin{dcases*}
      \tilde{\varSigma} \bullet (\varLambda + \varOmega) 
      + \beta v  
      & if $\| V \|_{1,1} \le v$, $\varOmega \succeq 0$, 
      $V = - B^{\top} (\varLambda + \varOmega) B$, \\
      +\infty 
      & otherwise. 
    \end{dcases*}
  \end{align*}
  Therefore, the dual problem in \eqref{eq:ell1.dual.problem} is 
  explicitly written as follows: 
  \begin{alignat*}{3}
    & \MIN_{v \in V, \, \varLambda \in \SC^{k}, \, \varOmega \in \SC^{k}}
    &{\quad}& 
    \tilde{\varSigma} \bullet (\varLambda + \varOmega) 
    + \beta v \\
    & \ST && 
    \| B^{\top} (\varLambda + \varOmega) B \|_{1,1} \le v , \\
      & &&
    \varOmega \succeq 0 .
  \end{alignat*}
  Constraint $\| B^{\top} (\varLambda + \varOmega) B \|_{1,1} \le v$ 
  becomes active at an optimal solution, which concludes the proof. 
\end{proof}

We are now in position to state the main result of this section. 
By using \refthm{thm:robust.general}, \reflmm{lmm:ell1.norm.mu}, and 
\reflmm{lmm:ell1.norm.varSigma}, we obtain the following fact. 

\begin{theorem}\label{thm:ell.infty.norm}
  Let $U_{\bi{\mu}}$ and $U _{\varSigma}$ be the sets defined by 
  \eqref{eq:uncertainty.set.ell.infty.1} and 
  \eqref{eq:uncertainty.set.ell.infty.2}, respectively. 
  Then, $\tilde{\bi{x}} \in X$ satisfies \eqref{eq:robust.reliable} 
  if and only if there exists a pair of 
  $z \in \Re$ and $W \in \SC^{n}$ satisfying 
  \begin{align}
    & g(\tilde{\bi{x}}) 
    + \nabla g(\tilde{\bi{x}})^{\top} \tilde{\bi{\mu}} 
    + \alpha \| A^{\top} \nabla g(\tilde{\bi{x}}) \|_{1} 
    + \tilde{\varSigma} \bullet W  
    + \beta \| B^{\top} W B \|_{1,1} 
    + \kappa^{2} z \le 0 , 
    \label{eq:ell.infty.norm.1} \\
    & 
    \begin{bmatrix}
      W & \nabla g(\tilde{\bi{x}})/2 \\
      \nabla g(\tilde{\bi{x}})^{\top}/2 & z \\
    \end{bmatrix}
    \succeq 0 . 
    \label{eq:ell.infty.norm.2}
  \end{align}
\end{theorem}
\begin{proof}
  It follows from \reflmm{lmm:ell1.norm.mu} and 
  \reflmm{lmm:ell1.norm.varSigma} that 
  \eqref{eq:robust.general.1} and \eqref{eq:robust.general.2} in 
  \refthm{thm:robust.general} are equivalently rewritten as 
  \begin{align*}
    & g(\tilde{\bi{x}}) 
    + \nabla g(\tilde{\bi{x}})^{\top} \tilde{\bi{\mu}} 
    + \alpha \| A^{\top} \nabla g(\tilde{\bi{x}}) \|_{1} 
    + \tilde{\varSigma} \bullet (\varLambda + \varOmega) 
    + \beta \| B^{\top} (\varLambda + \varOmega) B \|_{1,1} 
    + \kappa^{2} z \le 0 , \\
    & \varOmega \succeq 0 , \\
    & 
    \begin{bmatrix}
      \varLambda & \nabla g(\tilde{\bi{x}})/2 \\
      \nabla g(\tilde{\bi{x}})^{\top}/2 & z \\
    \end{bmatrix}
    \succeq 0 . 
  \end{align*}
  Put $W = \varLambda + \varOmega$ to see that this is reduced to 
  \begin{align*}
    & g(\tilde{\bi{x}}) 
    + \nabla g(\tilde{\bi{x}})^{\top} \tilde{\bi{\mu}} 
    + \alpha \| A^{\top} \nabla g(\tilde{\bi{x}}) \|_{1} 
    + \tilde{\varSigma} \bullet W  
    + \beta \| B^{\top} W B \|_{1,1} 
    + \kappa^{2} z \le 0 , \\
    & W - \varLambda \succeq 0 , \\
    & 
    \begin{bmatrix}
      \varLambda & \nabla g(\tilde{\bi{x}})/2 \\
      \nabla g(\tilde{\bi{x}})^{\top}/2 & z \\
    \end{bmatrix}
    \succeq 0 . 
  \end{align*}
  This is straightforwardly equivalent to \eqref{eq:ell.infty.norm.1} 
  and \eqref{eq:ell.infty.norm.2}. 
\end{proof}

It should be emphasized that 
\refthm{thm:ell.infty.norm} converts the set of infinitely many 
reliability constraints in \eqref{eq:robust.reliable.infinite} 
to two deterministic constraints, i.e., 
\eqref{eq:ell.infty.norm.1} and \eqref{eq:ell.infty.norm.2}. 
The latter constraints can be handled within the framework of 
conventional (deterministic) optimization.

\subsection{Uncertainty model with $\ell_{2}$-norm}
\label{sec:robust.ell.2}

In this section, we consider the uncertainty sets defined by 
\begin{align*}
  U_{\bi{\mu}} 
  &=  \{ \tilde{\bi{\mu}} + A \bi{z}_{1} 
  \mid
  \| \bi{z}_{1} \|_{2} \le \alpha , \
  \bi{z}_{1} \in \Re^{m}  \} , \\
  U_{\varSigma} 
  &= \{ \tilde{\varSigma} + B Z_{2} B^{\top} 
  \mid 
  \| Z_{2} \|_{\rr{F}} \le \beta , \
  Z_{2} \in \SC^{k} 
  \} \cap \SC_{+}^{n} . 
\end{align*}

\begin{example}
  As a simple example, put $\tilde{\bi{\mu}}=\bi{0}$ and $A=I$ with 
  $m=n$ to obtain 
  \begin{align*}
    U_{\bi{\mu}} 
    =  \{ \bi{z}_{1} \in \Re^{n} 
    \mid
    \| \bi{z}_{1} \|_{2} \le \alpha  \} . 
  \end{align*}
  This means that the expected value vector $\bi{\mu}$ belongs to a 
  hypersphere centered at the origin with radius $\alpha$. 
  Similarly, putting $B=I$ and $k=n$ we obtain 
  \begin{align*}
    U_{\varSigma} 
    = \{ \tilde{\varSigma} + Z_{2} 
    \mid 
    \| Z_{2} \|_{\rr{F}} \le \beta , \
    Z_{2}^{\top} = Z_{2} 
    \} \cap \SC_{+}^{n} . 
  \end{align*}
  This means that the variance-covariance matrix $\varSigma$ satisfies 
  \begin{align*}
    \sum_{i=1}^{n} \sum_{j=1}^{n} 
    (\varSigma_{ij} - \tilde{\varSigma}_{ij})^{2}  \le \beta^{2} 
  \end{align*}
  and is symmetric positive semidefinite. 
  \finbox
\end{example}

In a manner parallel to the proofs of 
\reflmm{lmm:ell1.norm.mu} and \reflmm{lmm:ell1.norm.varSigma}, we can 
obtain 
\begin{align*}
  \max_{\bi{\mu} \in U_{\bi{\mu}}} \{ 
  \nabla g(\tilde{\bi{x}})^{\top} \bi{\mu} 
  \} 
  &= \nabla g(\tilde{\bi{x}})^{\top} \tilde{\bi{\mu}} 
  + \alpha \| A^{\top} \nabla g(\tilde{\bi{x}}) \|_{2} , \\
  \max_{\varSigma \in U_{\varSigma}} \{
  \varLambda \bullet \varSigma  \} 
  &= \min_{\varOmega \in \SC^{k}} \{
  \tilde{\varSigma} \bullet (\varLambda + \varOmega) 
  + \beta \| B^{\top} (\varLambda + \varOmega) B \|_{\rr{F}} 
  \mid
  \varOmega \succeq 0 \} . 
\end{align*}
Here, the facts 
\begin{align*}
  \max_{\bi{t} \in \Re^{n}} \{ \bi{s}^{\top} \bi{t} 
  \mid \| \bi{t} \|_{2} \le 1 \}  
  = \| \bi{s} \|_{2} 
\end{align*}
and 
\begin{align*}
  \MoveEqLeft
  \{ (t,T) \in \Re \times \SC^{k}  
  \mid s t + S \bullet T \ge 0 \ 
  (\forall (s,S) : \| S \|_{\rr{F}} \le s) \}   \notag\\
  &= \{ (t,T) \in \Re \times \SC^{k} \mid \| T \|_{\rr{F}} \le t \} 
\end{align*}
have been used. 

Accordingly, analogous to \refthm{thm:ell.infty.norm}, we obtain the 
following conclusion: $\tilde{\bi{x}} \in X$ satisfies 
\eqref{eq:robust.reliable} if and only if there exists a pair of 
$z \in \Re$ and $W \in \SC^{n}$ satisfying 
\begin{align*}
  & g(\tilde{\bi{x}}) 
  + \nabla g(\tilde{\bi{x}})^{\top} \tilde{\bi{\mu}} 
  + \alpha \| A^{\top} \nabla g(\tilde{\bi{x}}) \|_{2} 
  + \tilde{\varSigma} \bullet W 
  + \beta \| B^{\top} W B \|_{\rr{F}} 
  + \kappa^{2} z \le 0 ,  \\
  & 
  \begin{bmatrix}
    W & \nabla g(\tilde{\bi{x}})/2 \\
    \nabla g(\tilde{\bi{x}})^{\top}/2 & z \\
  \end{bmatrix}
  \succeq 0 . 
\end{align*}

\subsection{Truss optimization under compliance constraint}
\label{sec:robust.compliance}

In this section, we present 
how the results established in the preceding sections 
can be employed for a specific RBDO problem. 
As a simple example, we consider a reliability constraint on the 
compliance under a static external load. 
We assume linear elasticity and small deformation. 

For ease of comprehension, consider design optimization of a truss. 
In this context, $x_{j}$ denotes the cross-sectional area of truss 
member $j$ $(j=1,\dots,n)$, where $n$ is the number of members. 
We attempt to minimize the structural volume of the truss, 
$\bi{c}^{\top}\bi{x}$, under the compliance constraint, 
where $c_{j}$ denotes the undeformed member length. 
Let $\pi(\bi{x})$ denote the compliance corresponding to a static 
external load. 
The first-order approximation of the compliance constraint is written as 
\begin{align*}
  \pi(\tilde{\bi{x}}) + \nabla\pi(\tilde{\bi{x}})^{\top} \bi{\zeta}
  \le \bar{\pi} , 
\end{align*}
where $\bar{\pi}$ $(>0)$ is a specified upper bound for the compliance. 
Accordingly, the design optimization problem to be solved is formulated 
as follows: 
\begin{subequations}\label{P:truss.compliance.0}%
  \begin{alignat}{3}
    & \MIN_{\tilde{\bi{x}}}
    &{\quad}& 
    \bi{c}^{\top} \tilde{\bi{x}} \\
    & \ST && 
    \tilde{\bi{x}} \ge \bar{\bi{x}} , \\
    & &&
    \rr{P}_{\NC(\bi{\mu},\varSigma)} 
    \{ \pi(\tilde{\bi{x}}) + \nabla \pi(\tilde{\bi{x}})^{\top} \bi{\zeta} 
    \le \bar{\pi} \} 
    \ge 1 - \epsilon  , 
    \quad \forall (\bi{\mu},\varSigma) \in U . 
  \end{alignat}
\end{subequations}
Here, the specified lower bound for the member cross-sectional area, 
denoted by $\bar{x}_{j}$ $(j=1,\dots,n)$, is positive, because in this 
paper we restrict ourselves to optimization problems without variation 
of structural topology. 

As for uncertainty sets of the moments, consider, for example, 
$U_{\bi{\mu}}$ and $U_{\varSigma}$ studied in 
section~\ref{sec:robust.ell.2}. 
For simplicity put $A=B=I$ so as to obtain 
\begin{align*}
  U_{\bi{\mu}} 
  &=  \{ \tilde{\bi{\mu}} + \bi{z}_{1} 
  \mid
  \| \bi{z}_{1} \|_{2} \le \alpha , \
  \bi{z}_{1} \in \Re^{n}  \} , \\
  U_{\varSigma} 
  &= \{ \tilde{\varSigma} + Z_{2} 
  \mid 
  \| Z_{2} \|_{\rr{F}} \le \beta , \
  Z_{2} \in \SC^{n} 
  \} \cap \SC_{+}^{n} . 
\end{align*}
From the result in section~\ref{sec:robust.ell.2}, we see that problem 
\eqref{P:truss.compliance.0} is equivalently rewritten as follows: 
\begin{subequations}\label{P:truss.compliance.1}%
  \begin{alignat}{3}
    & \MIN_{\tilde{\bi{x}}, \, z, \, W}
    &{\quad}& 
    \bi{c}^{\top} \tilde{\bi{x}} \\
    & \ST && 
    \tilde{\bi{x}} \ge \bar{\bi{x}} , \\
    & &&
    \pi(\tilde{\bi{x}}) 
    + \nabla \pi(\tilde{\bi{x}})^{\top} \tilde{\bi{\mu}} 
    + \alpha \| \nabla \pi(\tilde{\bi{x}}) \|_{2} 
    + \tilde{\varSigma} \bullet W 
    + \beta \| W \|_{\rr{F}} 
    + \kappa^{2} z \le \bar{\pi} ,  
    \label{P:truss.compliance.1.3} \\
    & &&
    \begin{bmatrix}
      W & \nabla \pi(\tilde{\bi{x}})/2 \\
      \nabla \pi(\tilde{\bi{x}})^{\top}/2 & z \\
    \end{bmatrix}
    \succeq 0 . 
    \label{P:truss.compliance.1.4}
  \end{alignat}
\end{subequations}
Here, $\tilde{\bi{x}} \in \Re^{n}$, $z \in \Re$, and $W \in \SC^{n}$ are 
variables to be optimized. 
It is worth noting that problem \eqref{P:truss.compliance.1} is a 
nonlinear SDP problem. 

The remainder of this section is devoted to presenting a method for 
solving problem \eqref{P:truss.compliance.1} that will be used for the 
numerical experiments in section~\ref{sec:ex}. 

The method sequentially solves SDP problems that approximate 
problem \eqref{P:truss.compliance.1}, in a fashion similar to sequential 
SDP methods for nonlinear SDP problems \citep{YY15,KT06}. 
Let $\tilde{\bi{x}}^{k}$ denote the incumbent solution obtained at 
iteration $k-1$. 
Define $\bi{h}_{k} \in \Re^{n}$ by 
\begin{align*}
  \bi{h}_{k} = \nabla \pi(\tilde{\bi{x}}^{k}) . 
\end{align*}
At iteration $k$, we replace $\nabla \pi(\tilde{\bi{x}})$ in 
\eqref{P:truss.compliance.1.3} and \eqref{P:truss.compliance.1.4} with 
$\bi{h}_{k}$. 
Moreover, to deal with $\pi(\tilde{\bi{x}})$ in 
\eqref{P:truss.compliance.1.3}, we use the fact that $s \in \Re$ 
satisfies 
\begin{align*}
  \pi(\bi{x}) \le s
\end{align*}
if and only if 
\begin{align}
  \begin{bmatrix}
    K(\tilde{\bi{x}}) & \bi{p} \\
    \bi{p}^{\top} & s \\
  \end{bmatrix}
  \succeq 0 
  \label{eq:compliance.matrix.inequality}
\end{align}
is satisfied \citep[section~3.1]{Kan11}, where 
$K(\tilde{\bi{x}}) \in \SC^{d}$ is the stiffness matrix of the truss, 
$\bi{p} \in \Re^{d}$ is the external load vector, and 
$d$ is the number of degrees of freedom of the nodal displacements. 
It is worth noting that, for trusses, $K(\tilde{\bi{x}})$ is linear in 
$\tilde{\bi{x}}$. 
Therefore, \eqref{eq:compliance.matrix.inequality} is a linear matrix 
inequality with respect to $\tilde{\bi{x}}$ and $s$, and hence can be 
handled within the framework of (linear) SDP. 
By this means, we obtain the following subproblem that is solved at 
iteration $k$ for updating $\tilde{\bi{x}}^{k}$ to $\tilde{\bi{x}}^{k+1}$: 
\begin{subequations}\label{P.sequential.SDP.sub.problem}%
  \begin{alignat}{3}
    & \MIN_{\tilde{\bi{x}}, \, z, \, w, \, s}
    &{\quad}& 
    \bi{c}^{\top} \tilde{\bi{x}} \\
    & \st && 
    \tilde{\bi{x}} \ge \bar{\bi{x}} , \\
    & && 
    \begin{bmatrix}
      k(\tilde{\bi{x}}) & \bi{p} \\
      \bi{p}^{\top} & s \\
    \end{bmatrix}
    \succeq 0 , \\
    & &&
    s + \bi{h}_{k}^{\top} \tilde{\bi{\mu}} 
    + \alpha \| \bi{h}_{k} \|_{2} 
    + \tilde{\varSigma} \bullet w 
    + \beta \| w \|_{\rr{F}} 
    + \kappa^{2} z \le \bar{\pi} , \\
    & &&
    \begin{bmatrix}
      w & \bi{h}_{k}/2 \\
      \bi{h}_{k}^{\top}/2 & z \\
    \end{bmatrix}
    \succeq 0 . 
  \end{alignat}
\end{subequations}
Since this is a linear SDP problem, we can solve this problem 
efficiently with a primal-dual interior-point method \citep{AL12}.

\section{Extensions}
\label{sec:general}

This section discusses some extensions of the results obtained in 
section~\ref{sec:robust}. 

\subsection{Robustness against uncertainty in distribution type}
\label{sec:general.non-Gaussian}

An important extension is that the obtained results can be applied to 
the case that, not only the moments, but also the type of probability 
distribution are unknown. 
In this case, we consider any combination of all types of probability 
distributions and all possible moments (expected value vectors and 
variance-covariance matrices) in the uncertainty set, and require that 
the failure probability is no greater than a specified value. 
This robustness against uncertainty in distribution type is important as 
the input distribution in practice is not necessarily known to be a 
normal distribution. 

Recall that, in section~\ref{sec:fundamental} and 
section~\ref{sec:robust}, we assumed that the design variables, $\bi{x}$, 
follows a normal distribution. 
Then we consider the robust reliability constraint in 
\eqref{eq:robust.reliable}. 
For the sake of clarity, we restate this problem setting in a slightly 
different manner. 
We have assumed that random vector $\bi{\zeta}$ can possibly follow any 
normal distribution satisfying 
$\bi{\mu} \in U_{\bi{\mu}}$ and $\varSigma \in U_{\varSigma}$. 
We use $\PC_{\NC}$ to denote the set of such normal distributions, i.e., 
\begin{align}
  \PC_{\NC} = 
  \{  \NC(\bi{\mu},\varSigma) 
  \mid
  \bi{\mu} \in U_{\bi{\mu}} , \
  \varSigma \in U_{\varSigma}  \} . 
  \label{eq:set.of.all.normal}
\end{align}
In other words, $\PC_{\NC}$ is the set of all possible realizations of 
the input distribution. 
We write $p \in \PC_{\NC}$ if $p$ is one of such realizations. 
With this new notation, \eqref{eq:robust.reliable} can be rewritten 
equivalently as 
\begin{align}
  \sup_{p \in \PC_{\NC}} \bigl\{
  \rr{P}_{p} 
  \{ g(\tilde{\bi{x}}) + \nabla g(\tilde{\bi{x}})^{\top} \bi{\zeta}  \le 0 \} 
  \bigr\}  \ge 1 - \epsilon .  
  \label{eq:nominal.robust.reliable}
\end{align}
For $U_{\bi{\mu}}$ and $U_{\varSigma}$ defined in  
section~\ref{sec:robust.ell.infty}, \refthm{thm:ell.infty.norm} 
shows that \eqref{eq:nominal.robust.reliable} is equivalent to 
\eqref{eq:ell.infty.norm.1} and \eqref{eq:ell.infty.norm.2}. 

We are now in position to consider {\em any\/} type of probability 
distribution. 
Only what we assume is that the input distribution satisfies 
$\bi{\mu} \in U_{\bi{\mu}}$ and $\varSigma \in U_{\varSigma}$, where, 
for a while, we consider $U_{\bi{\mu}}$ and $U_{\varSigma}$ defined in 
section~\ref{sec:robust.ell.infty}. 
We use $\PC$ to denote the set of such distributions, i.e., 
\begin{align}
  \PC = 
  \{ p \mid
  p \text{ is a distribution with 
  $\bi{\mu} \in U_{\bi{\mu}}$ \& 
  $\varSigma \in U_{\varSigma}$} 
  \} . 
  \label{eq:set.of.all.distributions}
\end{align}
Then, instead of \eqref{eq:nominal.robust.reliable}, we consider the 
following constraint: 
\begin{align}
  \sup_{p \in \PC} \bigl\{
  \rr{P}_{p} 
  \{ g(\tilde{\bi{x}}) + \nabla g(\tilde{\bi{x}})^{\top} \bi{\zeta}  \le 0 \} 
  \bigr\}  \ge 1 - \epsilon .  
  \label{eq:all.robust.reliable}
\end{align}
That is, we require that the reliability constraint should be satisfied 
for any input distribution $p$ satisfying $p \in \PC$. 
A main assertion of this section is that, by simply setting 
\begin{align}
  \kappa = \sqrt{\frac{1-\epsilon}{\epsilon}}  
  \label{eq:kappa.all.distributions}
\end{align}
instead of $\kappa=-\varPhi^{-1}(\epsilon)$, constraint 
\eqref{eq:all.robust.reliable} is equivalent to 
\eqref{eq:ell.infty.norm.1} and \eqref{eq:ell.infty.norm.2} in 
\refthm{thm:ell.infty.norm}. 
We can show this fact in the following manner. 
Let $\PC(\bi{\mu},\varSigma)$ denote the set of probability 
distributions, the expected value vector and the variance-covariance 
matrix of which are $\bi{\mu}$ and $\varSigma$, respectively. 
Observe that, with $\PC(\bi{\mu},\varSigma)$, 
\eqref{eq:all.robust.reliable} can be rewritten equivalently as 
\begin{align}
  \sup_{(\bi{\mu},\varSigma) \in U} \left\{
  \sup_{p \in \PC(\bi{\mu},\varSigma)} \bigl\{
  \rr{P}_{p} 
  \{ g(\tilde{\bi{x}}) + \nabla g(\tilde{\bi{x}})^{\top} \bi{\zeta}  \le 0 \} 
  \bigr\} \right\} \ge 1 - \epsilon . 
  \label{eq:all.robust.reliable.2}
\end{align}
With relation to the inner supremum, consider the condition 
\begin{align}
  \sup_{p \in \PC(\bi{\mu},\varSigma)} \bigl\{
  \rr{P}_{p} 
  \{ g(\tilde{\bi{x}}) + \nabla g(\tilde{\bi{x}})^{\top} \bi{\zeta}  \le 0 \} 
  \bigr\}  \ge 1 - \epsilon . 
  \label{eq:all.robust.reliable.3}
\end{align}
\citet[Theorem~1]{EOO03} show that \eqref{eq:all.robust.reliable.3} 
holds if and only if \eqref{eq:reliability.constraint.3} of 
\refthm{thm:reliability.constraint} holds with $\kappa$ defined by 
\eqref{eq:kappa.all.distributions}. 
Therefore, all the subsequent results established in 
section~\ref{sec:fundamental} and section~\ref{sec:robust} hold by 
simply replacing the value of $\kappa$ with the one in 
\eqref{eq:kappa.all.distributions}. 
Thus, the robust reliability constraint with unknown distribution type 
is also reduced to the form in \eqref{eq:ell.infty.norm.1} and 
\eqref{eq:ell.infty.norm.2} of \refthm{thm:ell.infty.norm}. 

The result in section~\ref{sec:robust.ell.2}, which are established for 
the $\ell_{2}$-norm uncertainty model, is also extended to the case of 
unknown distribution type by replacing $\kappa$ with the value in 
\eqref{eq:kappa.all.distributions}.

\subsection{Multiple constraints}
\label{sec:general.multiple.constraints}

In section~\ref{sec:fundamental} and section~\ref{sec:robust}, 
we have restricted ourselves to the case that the design optimization 
problem has a single performance requirement, 
\eqref{eq:performance.constraint}. 
In this section, we discuss treatment of multiple constraints. 

Suppose that the performance requirement is written as 
\begin{align*}
  g_{i}(\bi{x}) \le 0 , 
  \quad
  i=1,\dots,m . 
\end{align*}
The first-order approximation yields 
\begin{align*}
  g_{i}^{\rr{lin}}(\bi{x}) \le 0 , 
  \quad
  i=1,\dots,m , 
\end{align*}
where 
$g_{i}^{\rr{lin}}(\bi{x}) 
  = g_{i}(\tilde{\bi{x}}) 
  + \nabla g_{i}(\tilde{\bi{x}})^{\top}\bi{\zeta}$ $(i=1,\dots,m)$. 
Suppose that we impose a distributionally-robust reliability constraint 
for each $i=1,\dots,m$ independently, i.e., 
\begin{align}
  \sup_{p \in \PC} \bigl\{
  \rr{P}_{p} 
  \{ g_{i}^{\rr{lin}}(\bi{x}) \le 0 \} 
  \bigr\}  \ge 1 - \epsilon ,  
  \quad
  i=1,\dots,m . 
  \label{eq:multiple.constraint.1}
\end{align}
Here, $\PC$ is the set of possible realizations of the input 
distribution (i.e., $\PC$ here is either $\PC_{\NC}$ in 
\eqref{eq:set.of.all.normal} or $\PC$ in 
\eqref{eq:set.of.all.distributions}). 
It is worth noting that in \eqref{eq:multiple.constraint.1} the worst 
case distributions are considered independently for each $i=1,\dots,m$. 
Constraint \eqref{eq:multiple.constraint.1} can be 
straightforwardly dealt with in the same manner as 
section~\ref{sec:robust}. 

In contrast, suppose that we consider a single (i.e., common) worst-case 
distribution for all $i=1,\dots,m$. 
Then the distributionally-robust reliability constraint is written as 
\begin{align}
  \sup_{p \in \PC} \bigl\{
  \rr{P}_{p} 
  \{ g_{i}^{\rr{lin}}(\bi{x}) \le 0 \
  (i=1,\dots,m) \} 
  \bigr\}  \ge 1 - \epsilon . 
  \label{eq:multiple.constraint.2}
\end{align}
Treatment of this constraint remains to be studied as future work. 
It is worth noting that constraint \eqref{eq:multiple.constraint.1} is 
conservative compared with constraint \eqref{eq:multiple.constraint.2}.

\section{Numerical examples}
\label{sec:ex}

In section~\ref{sec:robust.compliance} we have seen that an optimization 
problem of trusses under the compliance constraint is reduced to 
problem \eqref{P:truss.compliance.1}. 
In this section we solve this optimization problem numerically. 

The algorithm presented in section~\ref{sec:robust.compliance} 
was implemented in Matlab ver.~9.8.0.\footnote{%
Source codes for solving the optimization problems presented in 
section~\ref{sec:ex} are available on-line at 
\url{https://github.com/ykanno22/moment_worst/}. } 
The SDP problem in \eqref{P.sequential.SDP.sub.problem} was solved by 
CVX ver.~2.2 \citep{GB08,CVX} with 
SeDuMi ver.~1.3.4 \citep{Stu99,Pol05}. 
Computation was carried out on 
a 2.6{\,}GHz Intel Core i7-9750H processor with 32{\,}GB RAM. 

\begin{figure}[tbp]
  \centering
  \includegraphics[scale=1.10]{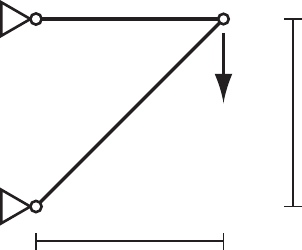}
  \begin{picture}(0,0)
    \put(-172,-38){
    \put(100,102){{\small $1$}}
    \put(100,76){{\small $2$}}
    \put(168,78){{\small $1\,\mathrm{m}$}}
    \put(106,31){{\small $1\,\mathrm{m}$}}
    }
  \end{picture}
  \bigskip
  \caption{Problem setting of example (I): 2-bar truss. }
  \label{fig:gs2bar}
\end{figure}

\begin{table*}[tbp]
  \centering
  \caption{Optimal solutions of example (I) 
  with $\epsilon=0.01$, $\alpha=0.2$, and $\beta=0.01$ 
  (probability distributions are assumed to be normal distributions). }
  \label{tab:solution.ex.I}
  \begin{tabular}{lrrrr}
    \toprule
    & $x_{1}$ ($\mathrm{mm}^2$) & $x_{2}$ ($\mathrm{mm}^2$) 
    & Obj.\ val.\ ($\mathrm{mm^{3}}$) & $\pi(\tilde{\bi{x}})$ (J) \\
    \midrule
    Nominal optim.\ 
    & 1500.0 & 2121.3 & $4.5000 \times 10^{6}$ & 100.000 \\
    $\ell_{\infty}$-norm unc.\
    & 1558.0 & 2203.4 & $4.6741 \times 10^{6}$ & 96.274 \\
    $\ell_{2}$-norm unc.\
    & 1535.4 & 2171.4 & $4.6063 \times 10^{6}$ & 97.692 \\
    \bottomrule
  \end{tabular}
\end{table*}

\begin{figure*}[tbp]
  \centering
  \subfloat[]{
  \label{fig:2bar_ell1_x1}
  \includegraphics[scale=0.40]{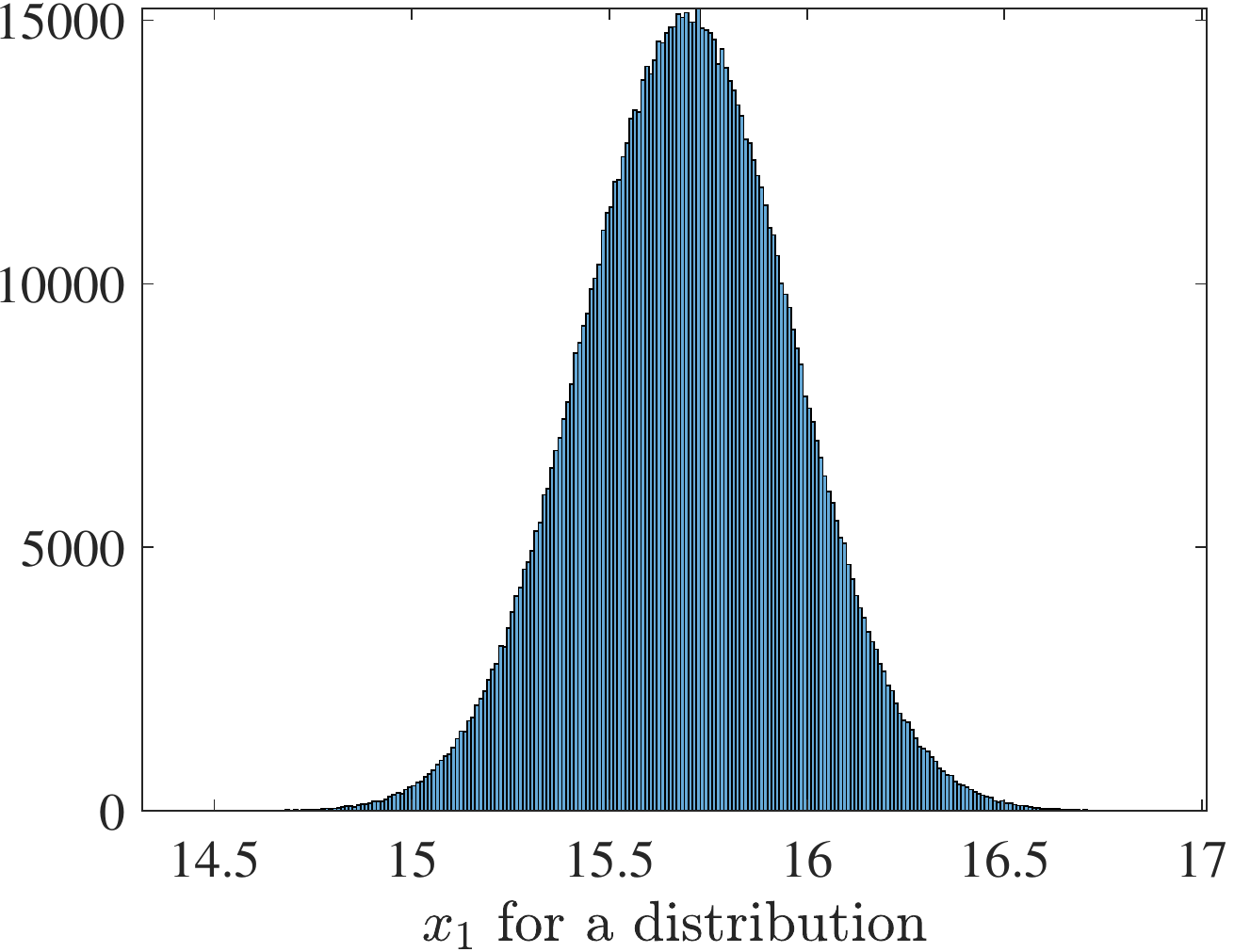}
  }
  \hfill
  \subfloat[]{
  \label{fig:2bar_ell1_x2}
  \includegraphics[scale=0.40]{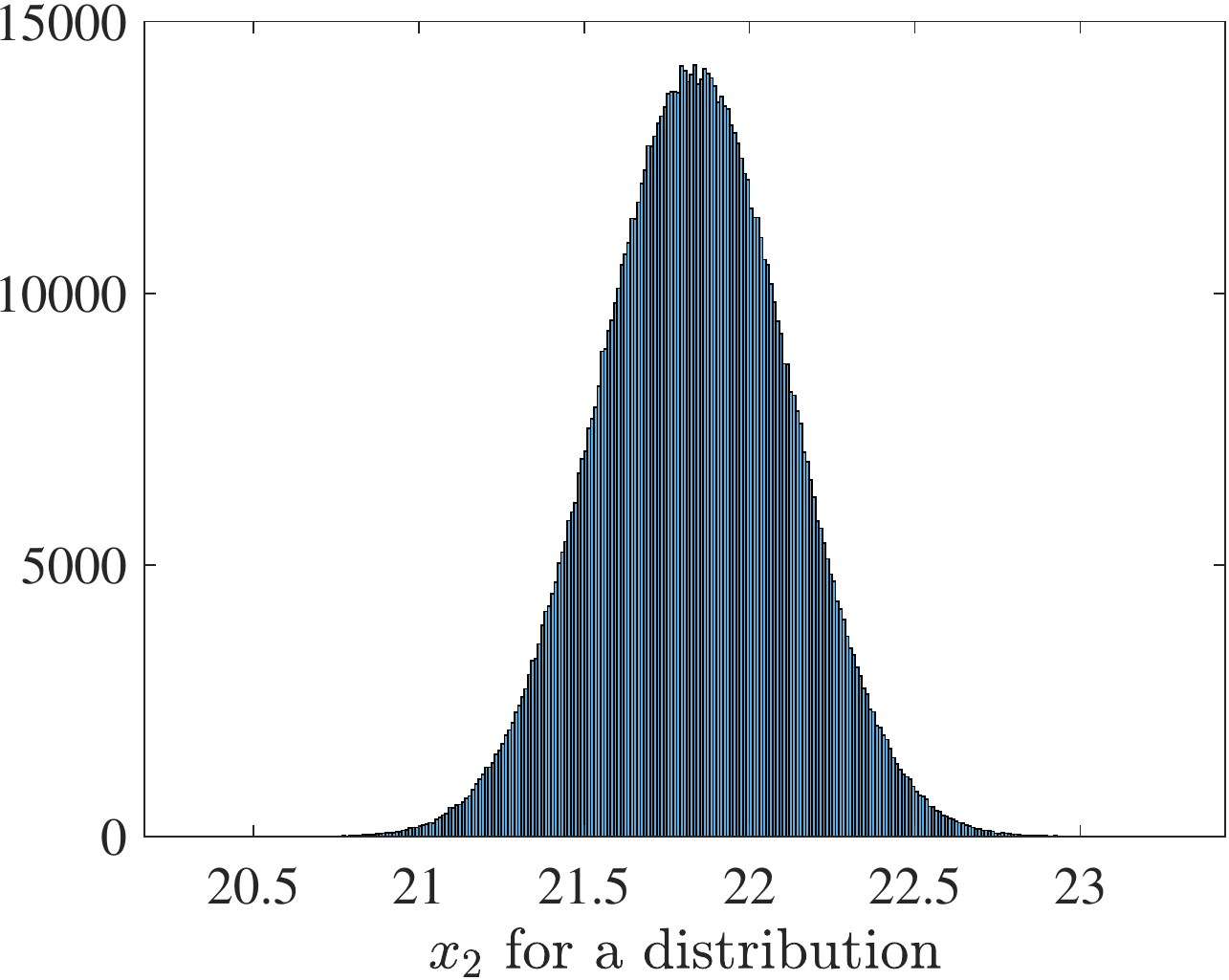}
  }
  \par
  \subfloat[]{
  \label{fig:2bar_ell1_sample_app_compl}
  \includegraphics[scale=0.40]{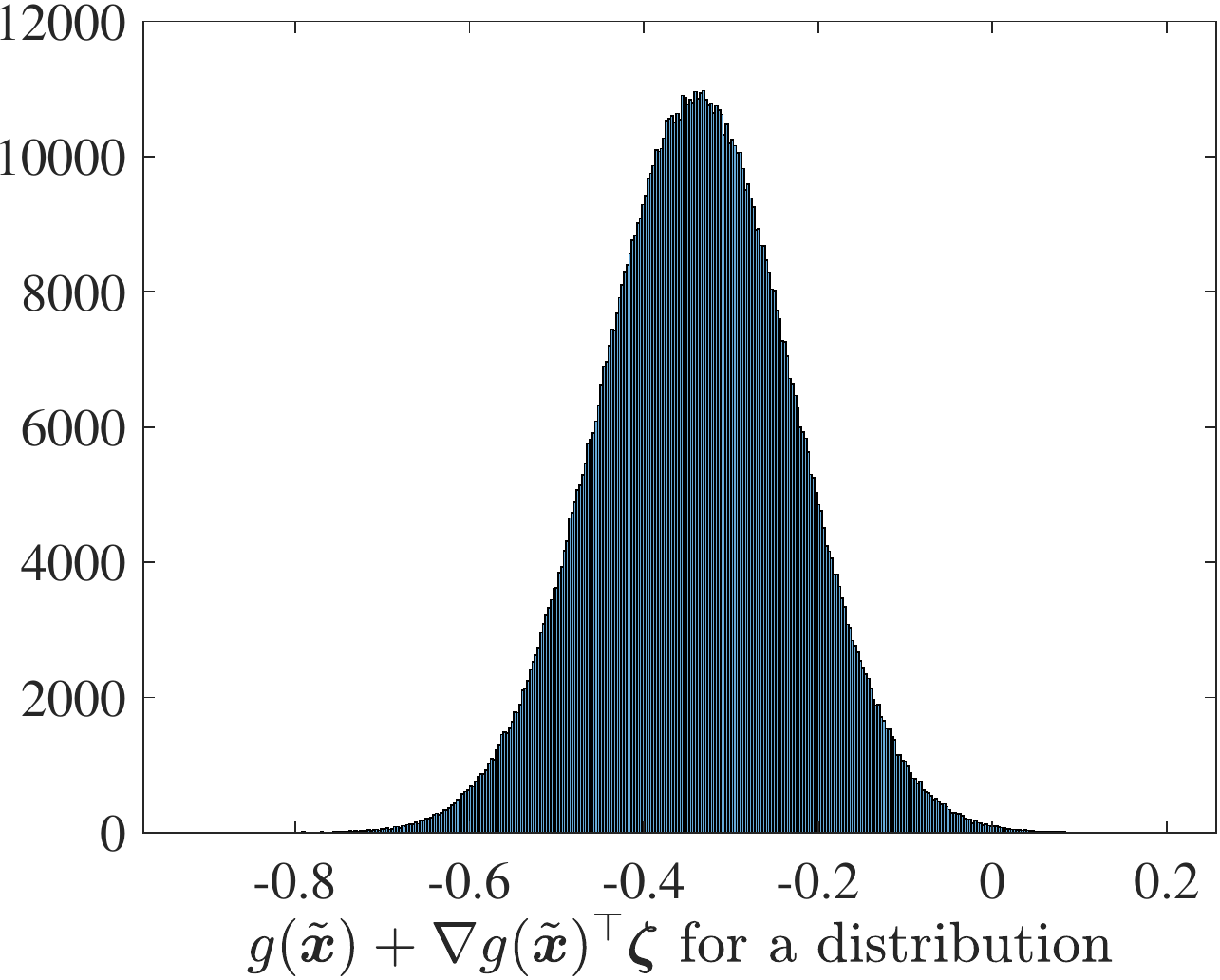}
  }
  \caption[]{Example of Monte Carlo simulation for a single sample of 
  $\bi{\mu}$ and $\varSigma$ (example (I) with the $\ell_{\infty}$-norm 
  uncertainty model; probability distributions are assumed 
  to be normal distributions). 
  \subref{fig:2bar_ell1_x1} 
  Samples of $x_{1}=\tilde{x}_{1}+\zeta_{1}$; 
  \subref{fig:2bar_ell1_x2} 
  samples of $x_{2}=\tilde{x}_{2}+\zeta_{2}$; and 
  \subref{fig:2bar_ell1_sample_app_compl} computed values of 
  $g(\tilde{\bi{x}}) + \nabla g(\tilde{\bi{x}})^{\top}\bi{\zeta}$.   }
\end{figure*}

\begin{figure*}[tbp]
  \centering
  \subfloat[]{
  \label{fig:2bar_ell1_prob_of_failure}
  \includegraphics[scale=0.40]{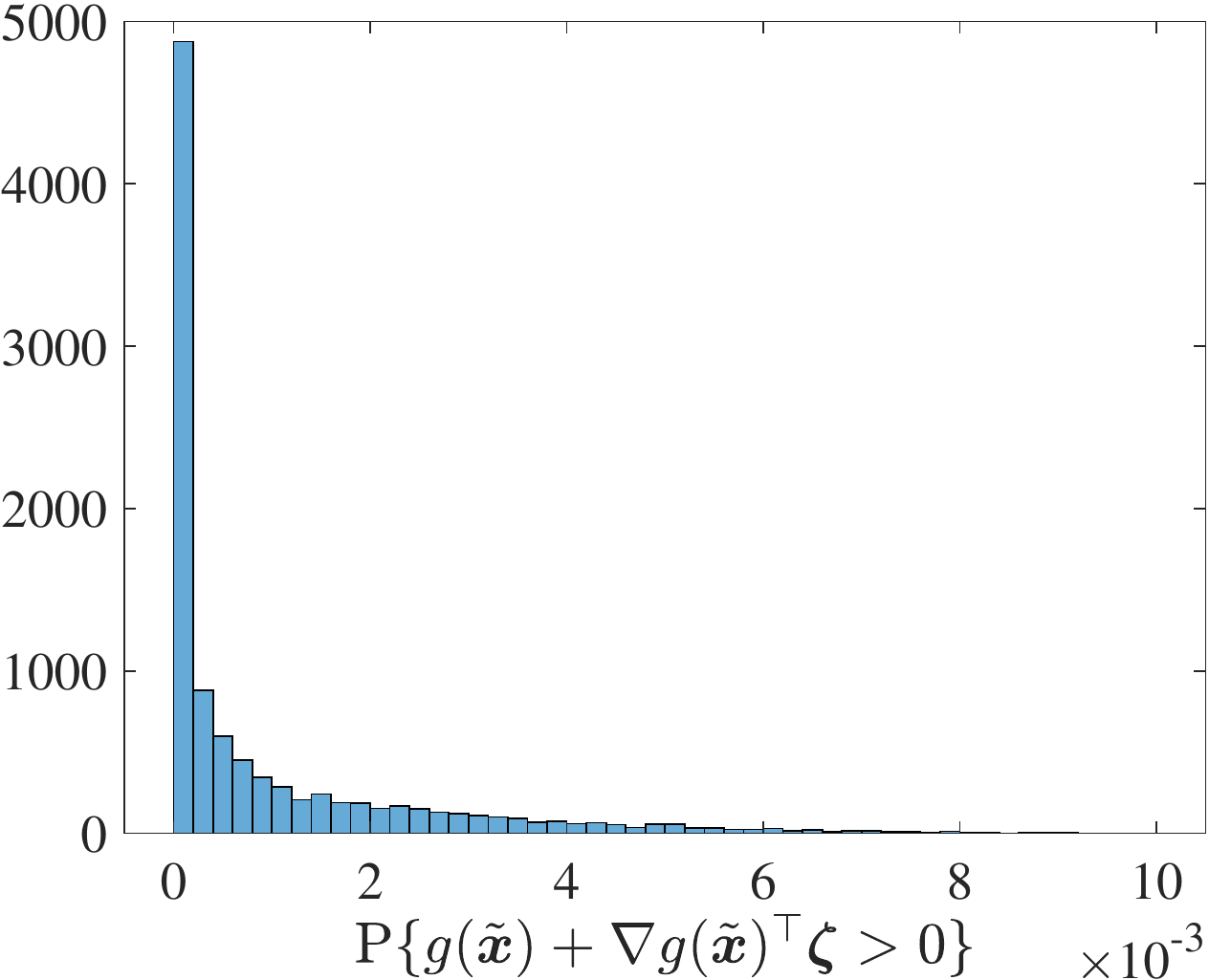}
  }
  \hfill
  \subfloat[]{
  \label{fig:2bar_ell1_exact_failure_prob}
  \includegraphics[scale=0.40]{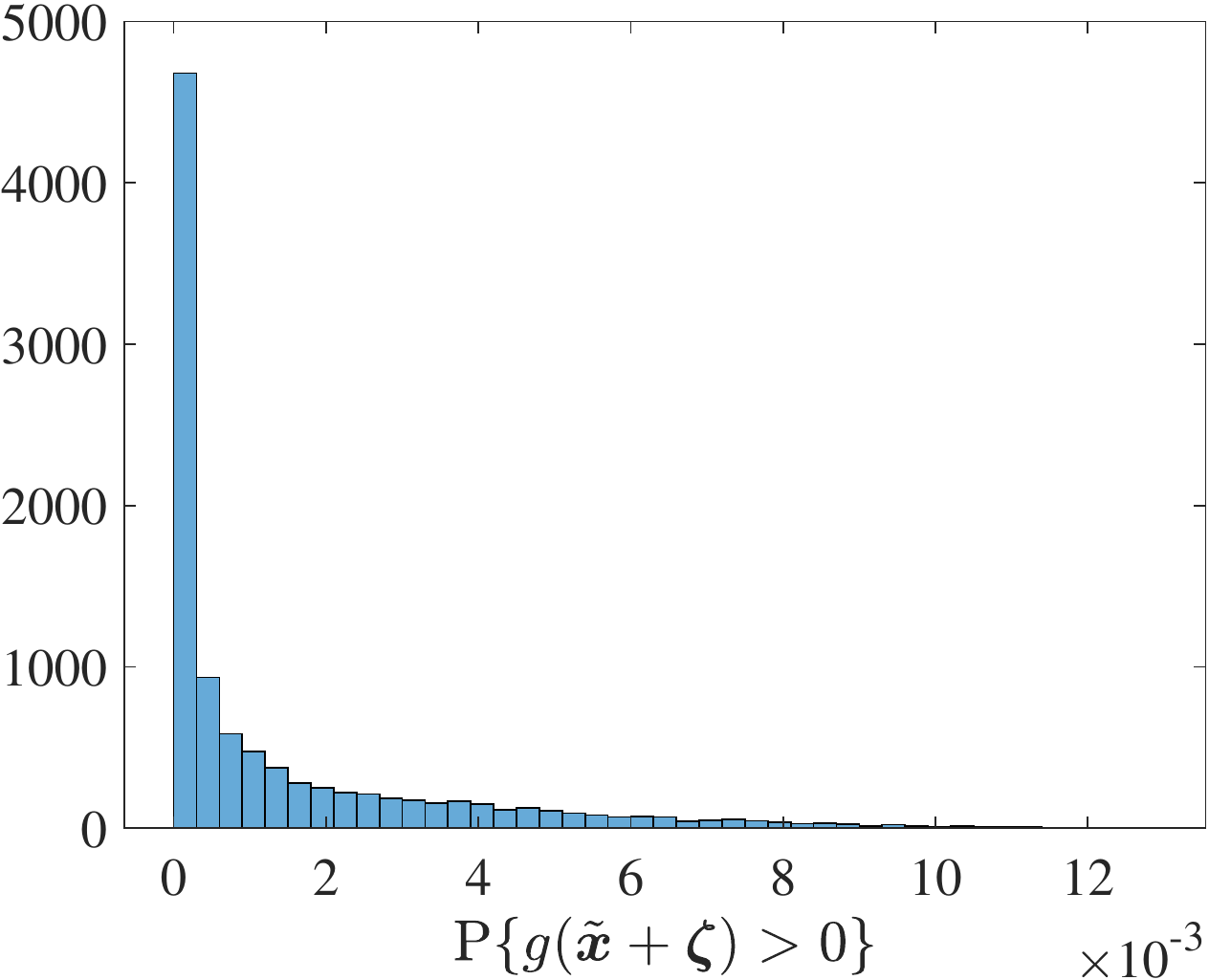}
  }
  \caption[]{Results of double-loop Monte Carlo simulation 
  (example (I) with the $\ell_{\infty}$-norm uncertainty model; probability 
  distributions are assumed to be normal distributions). 
  \subref{fig:2bar_ell1_prob_of_failure} Failure probability of linearly 
  approximated constraint 
  $g(\tilde{\bi{x}}) + \nabla g(\tilde{\bi{x}})^{\top}\bi{\zeta} \le 0$; and 
  \subref{fig:2bar_ell1_exact_failure_prob} failure probability of 
  constraint $g(\tilde{\bi{x}}+\bi{\zeta}) \le 0$ without approximation.   }
\end{figure*}

\begin{figure*}[tbp]
  \centering
  \subfloat[]{
  \label{fig:2bar_ell1_loop_reliability}
  \includegraphics[scale=0.40]{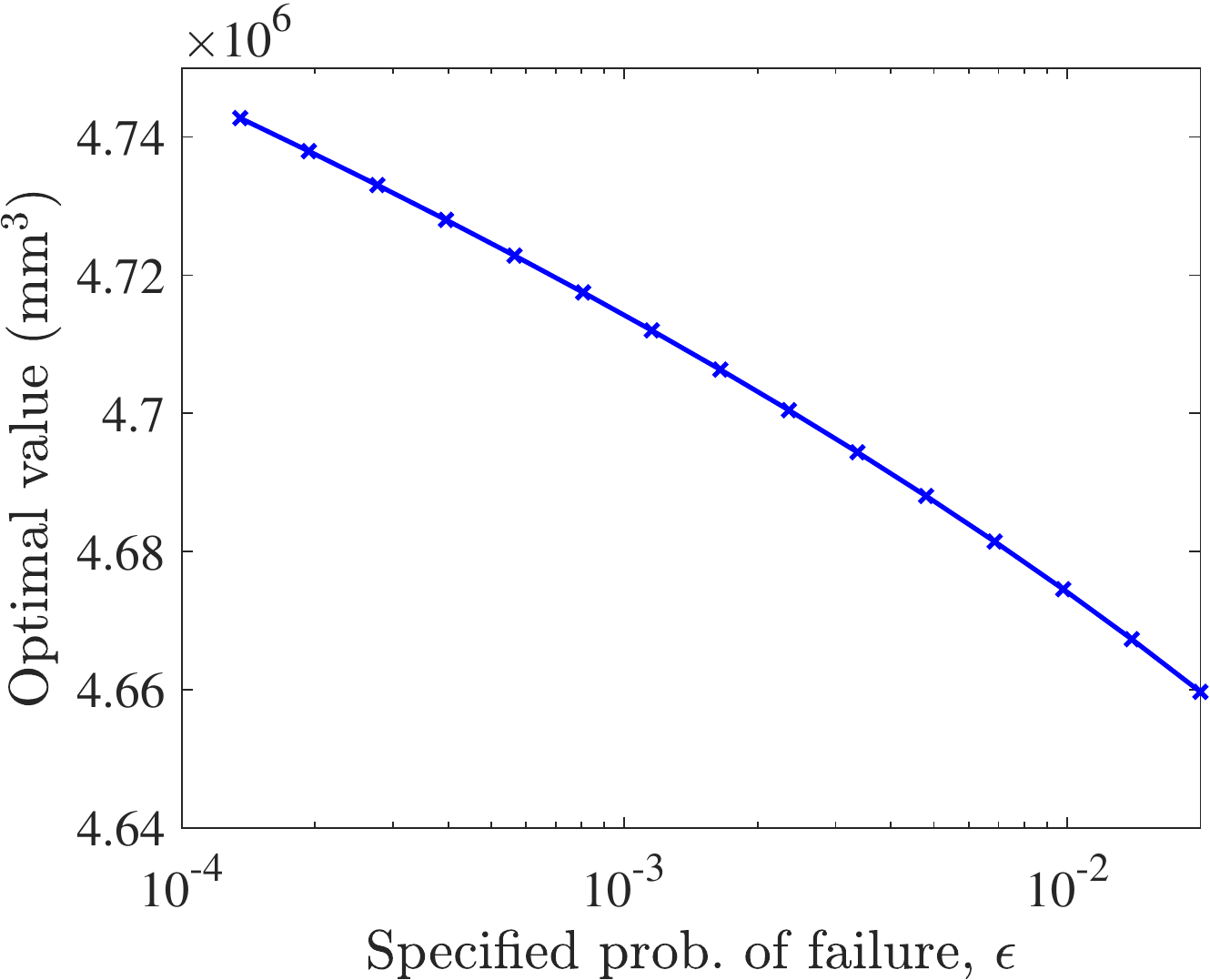}
  }
  \hfill
  \subfloat[]{
  \label{fig:2bar_ell1_loop_uncertainty}
  \includegraphics[scale=0.40]{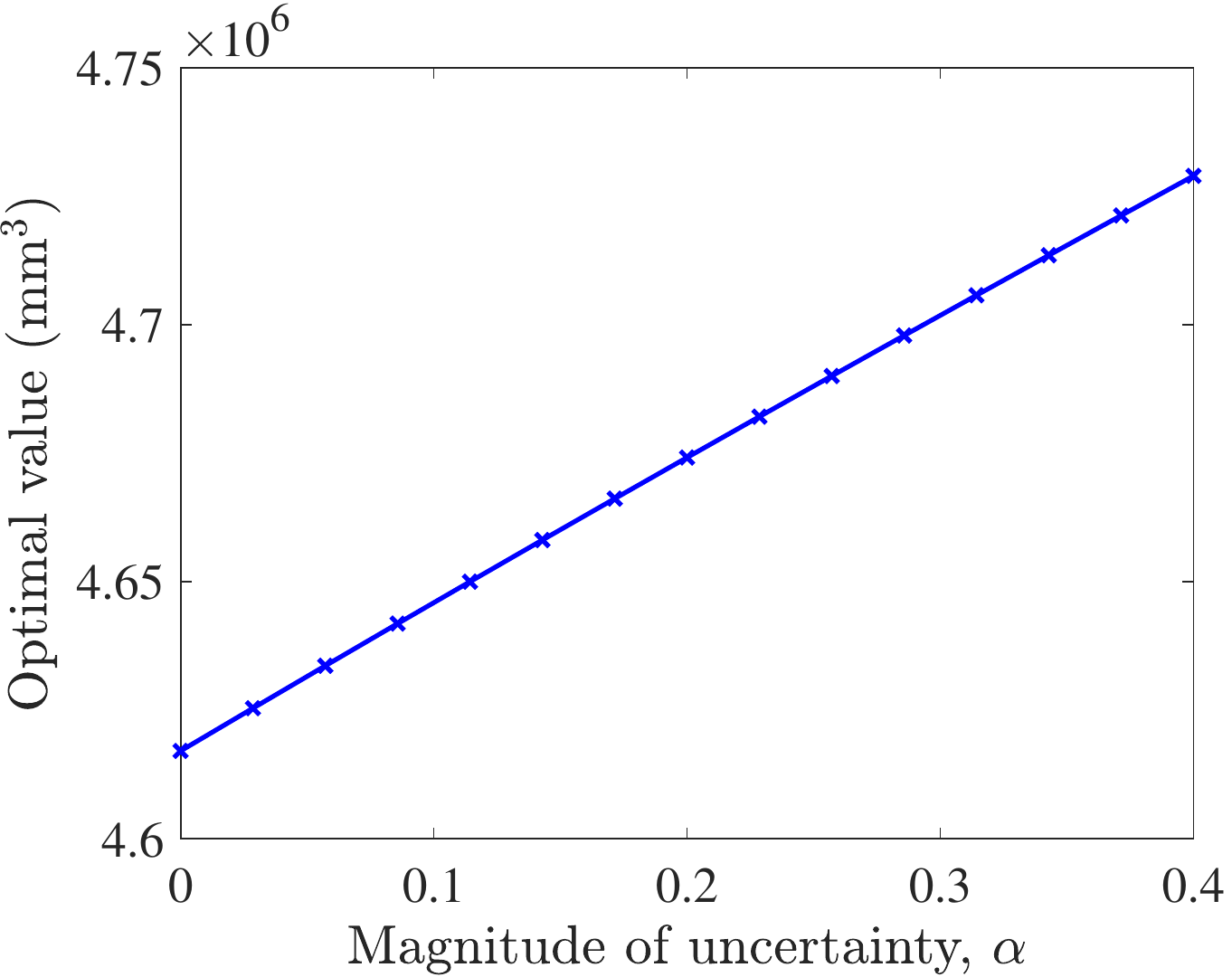}
  }
  \caption[]{Optimal value (example (I) with the $\ell_{\infty}$-norm 
  uncertainty model; probability distributions are assumed to be normal 
  distributions) versus 
  \subref{fig:2bar_ell1_loop_reliability} failure probability; and 
  \subref{fig:2bar_ell1_loop_uncertainty} magnitude of uncertainty.   }
  \label{fig:2bar_ell1_loop}
\end{figure*}

\begin{figure*}[tbp]
  \centering
  \includegraphics[scale=0.40]{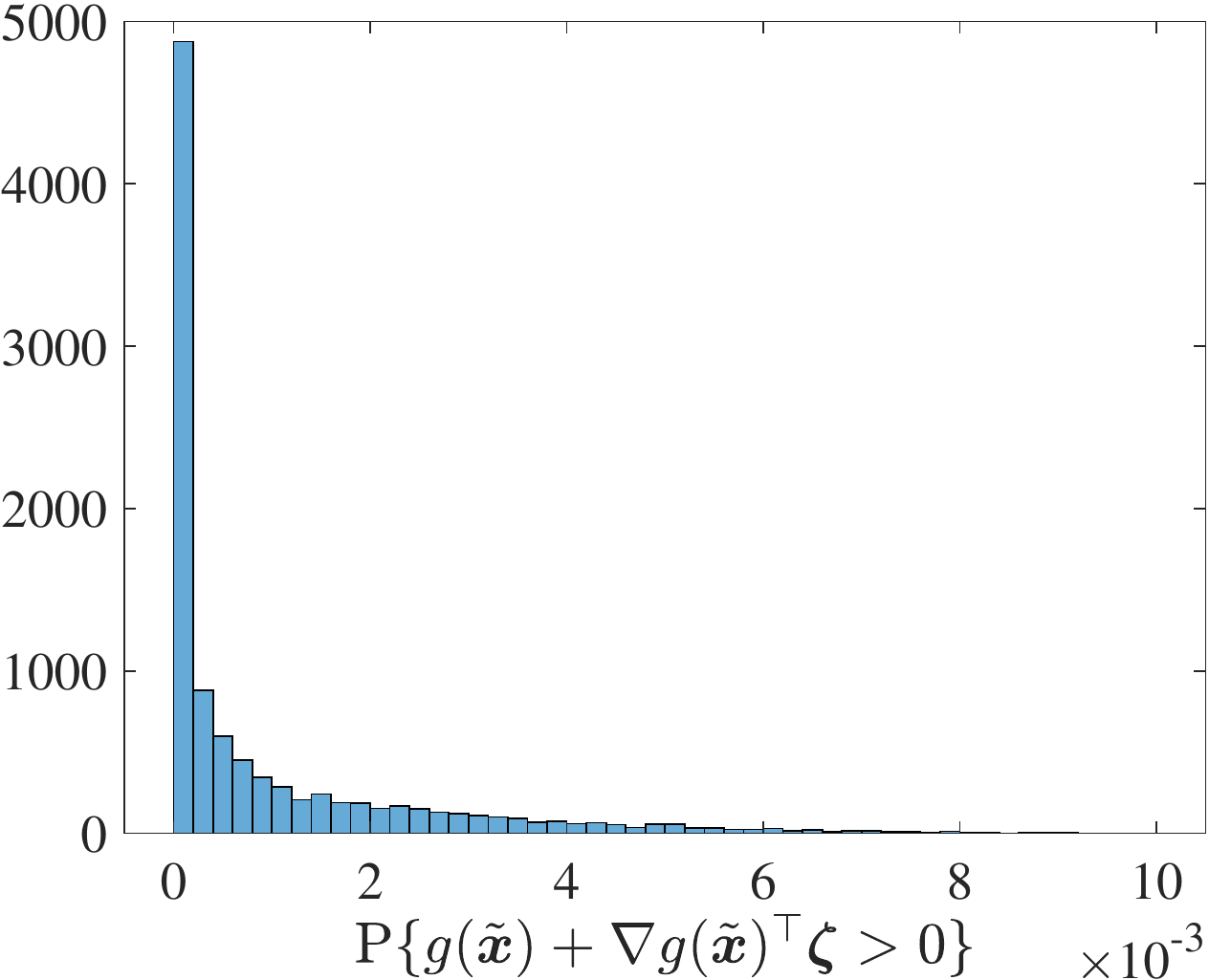}
  \caption[]{Results of double-loop Monte Carlo simulation 
  (example (I) with the $\ell_{2}$-norm uncertainty model; probability 
  distributions are assumed to be normal distributions). 
  Failure probability of linearly approximated constraint 
  $g(\tilde{\bi{x}}) + \nabla g(\tilde{\bi{x}})^{\top}\bi{\zeta} \le 0$.  }
  \label{fig:2bar_ell2_prob_of_failure}
\end{figure*}

\begin{figure*}[tbp]
  \centering
  \subfloat[]{
  \label{fig:2bar_ell2_loop_reliability}
  \includegraphics[scale=0.40]{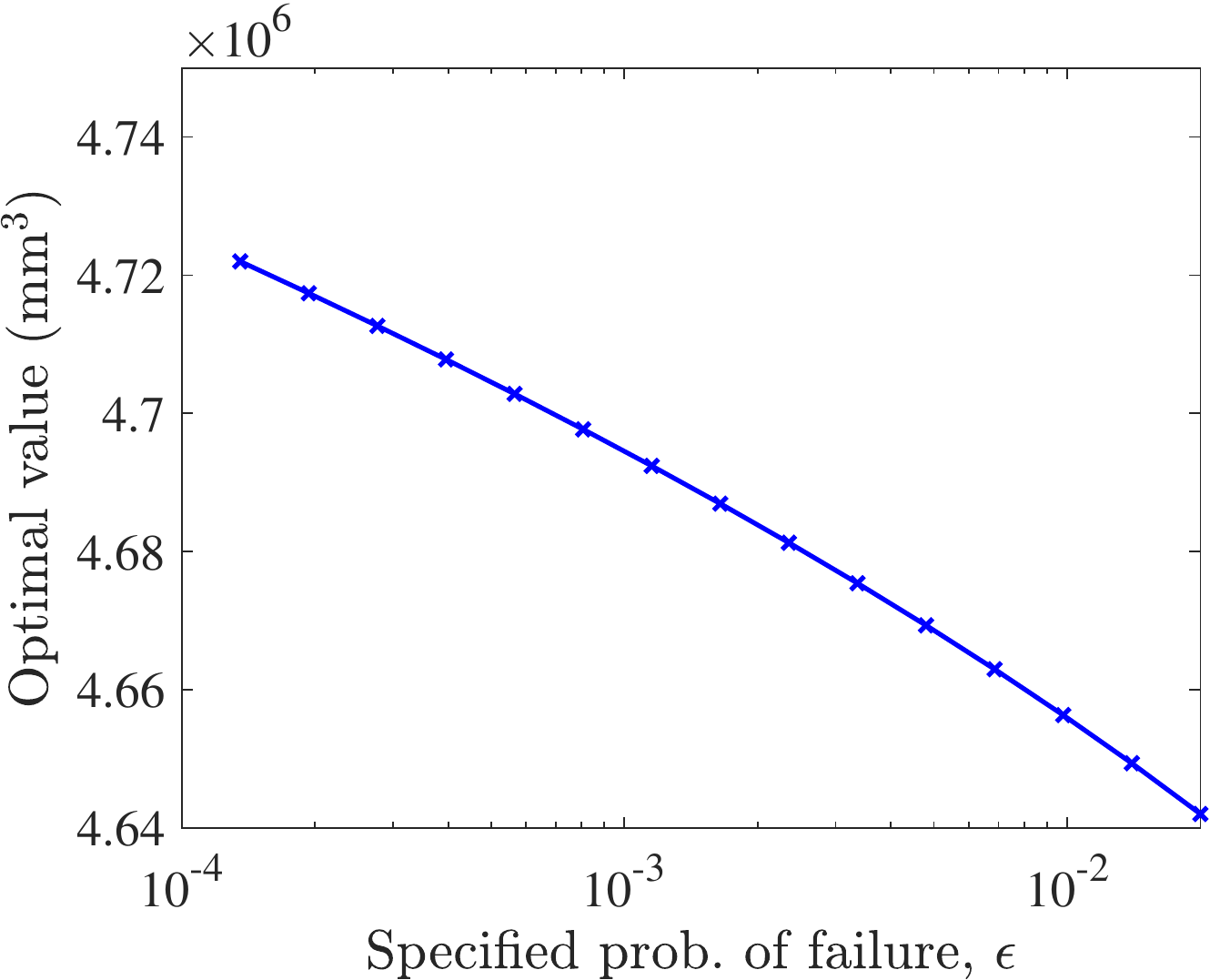}
  }
  \hfill
  \subfloat[]{
  \label{fig:2bar_ell2_loop_uncertainty}
  \includegraphics[scale=0.40]{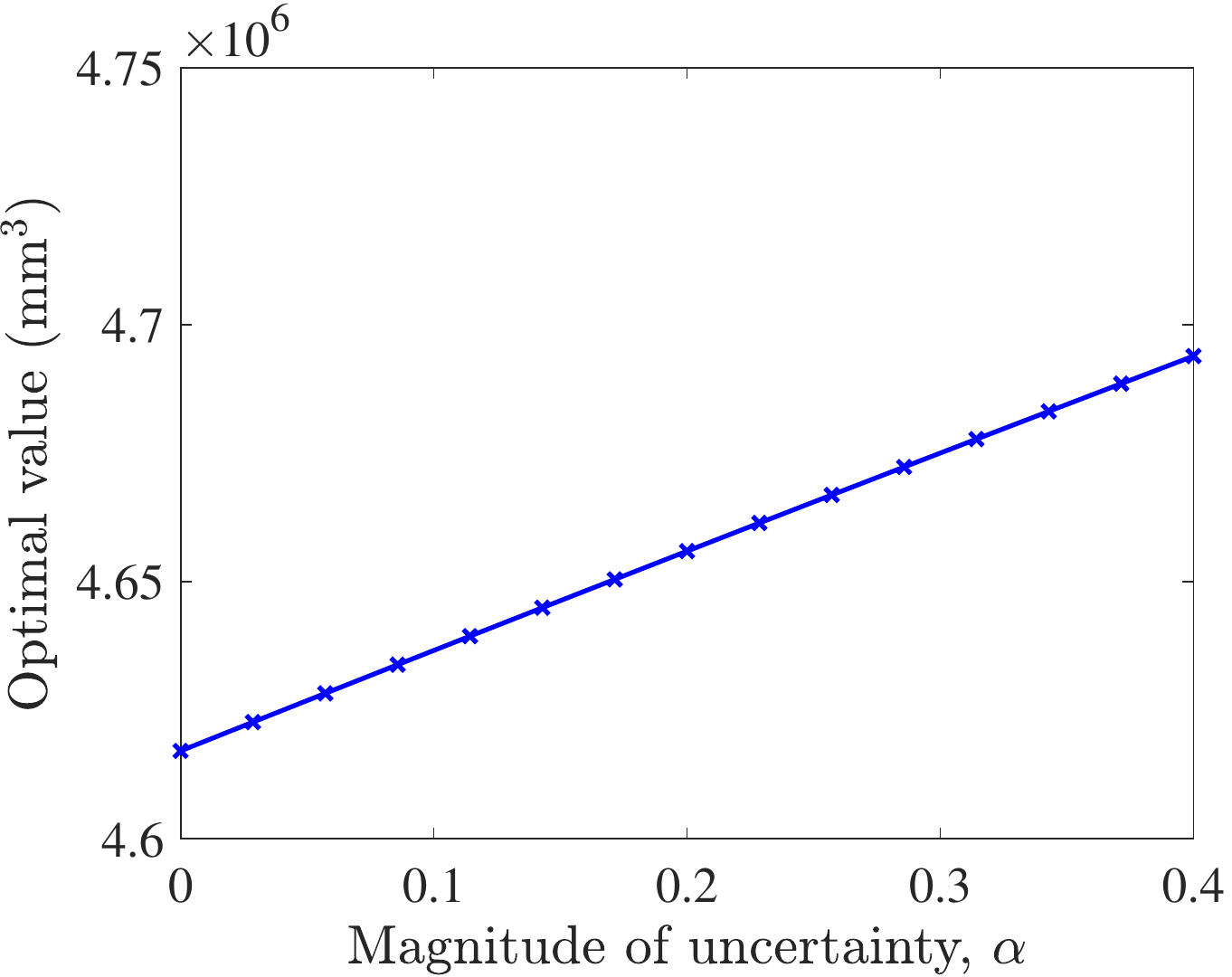}
  }
  \caption[]{Optimal value (example (I) with the $\ell_{2}$-norm 
  uncertainty model; probability distributions are assumed to be normal 
  distributions) versus 
  \subref{fig:2bar_ell2_loop_reliability}  failure probability; and 
  \subref{fig:2bar_ell2_loop_uncertainty} magnitude of uncertainty.   }
  \label{fig:2bar_ell2_loop}
\end{figure*}

\begin{figure*}[tbp]
  \centering
  \subfloat[]{
  \label{fig:2bar_ell1_all_dist_reliability}
  \includegraphics[scale=0.40]{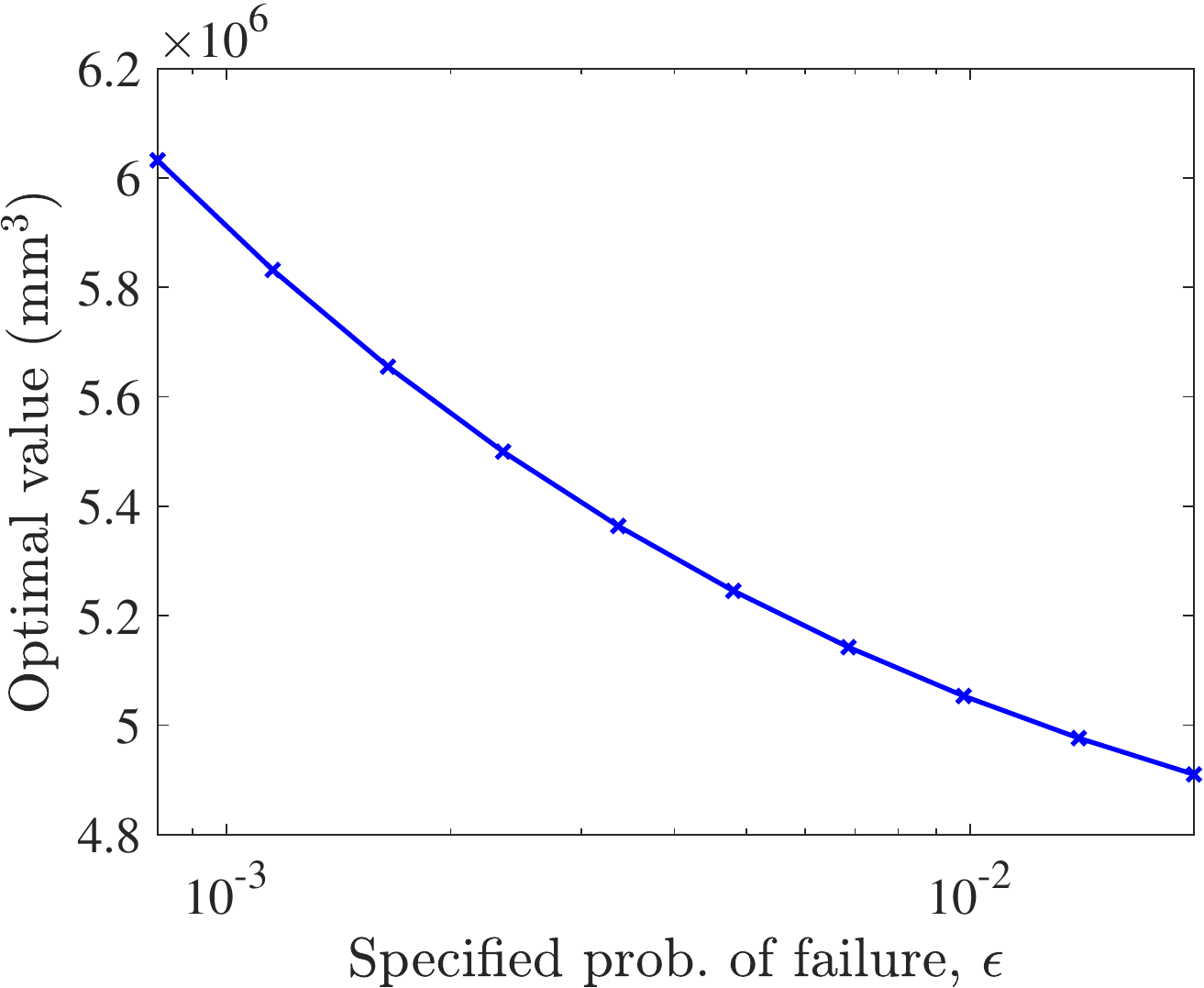}
  }
  \hfill
  \subfloat[]{
  \label{fig:2bar_ell1_all_dist_uncertainty}
  \includegraphics[scale=0.40]{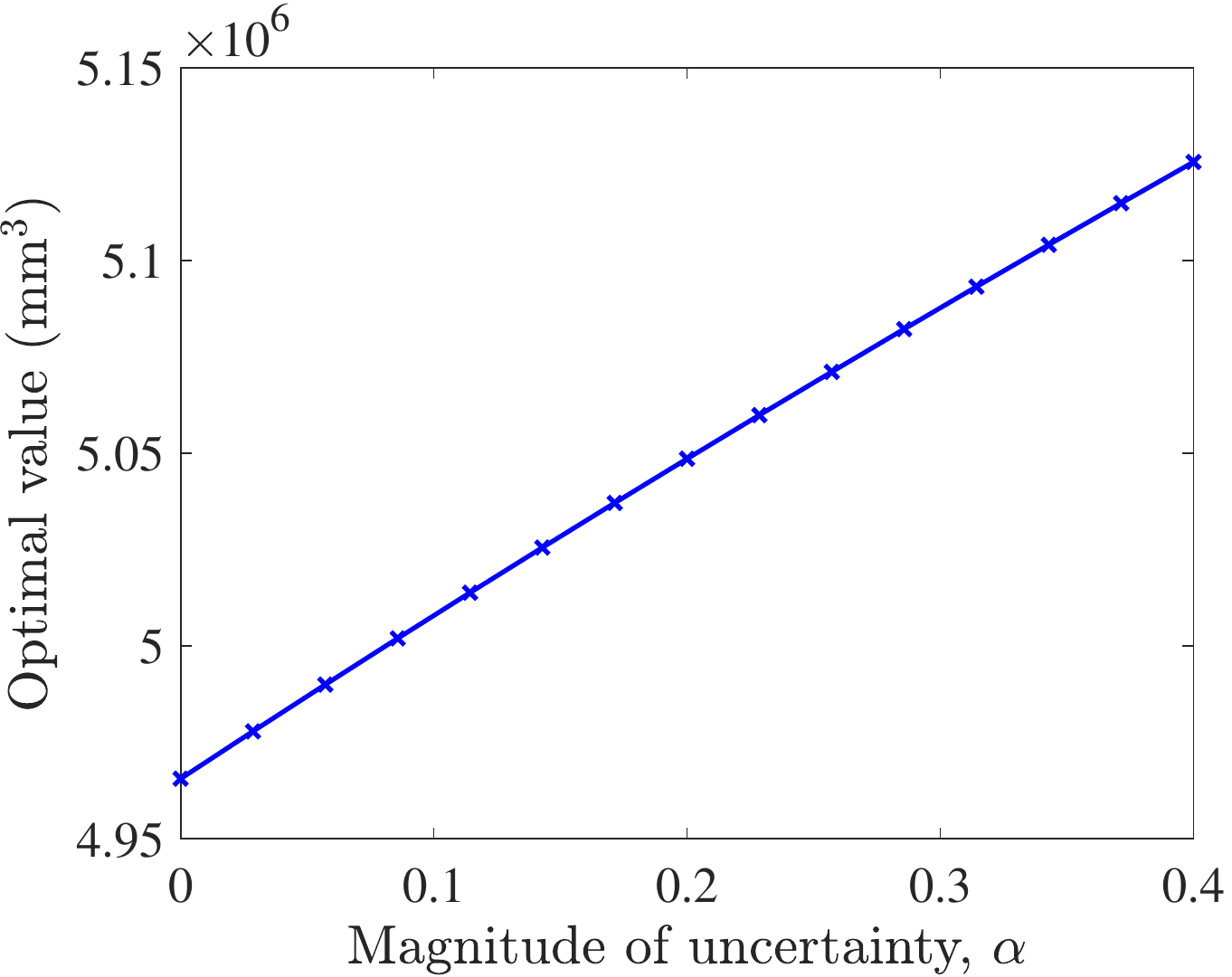}
  }
  \par
  \subfloat[]{
  \label{fig:2bar_ell2_all_dist_reliability}
  \includegraphics[scale=0.40]{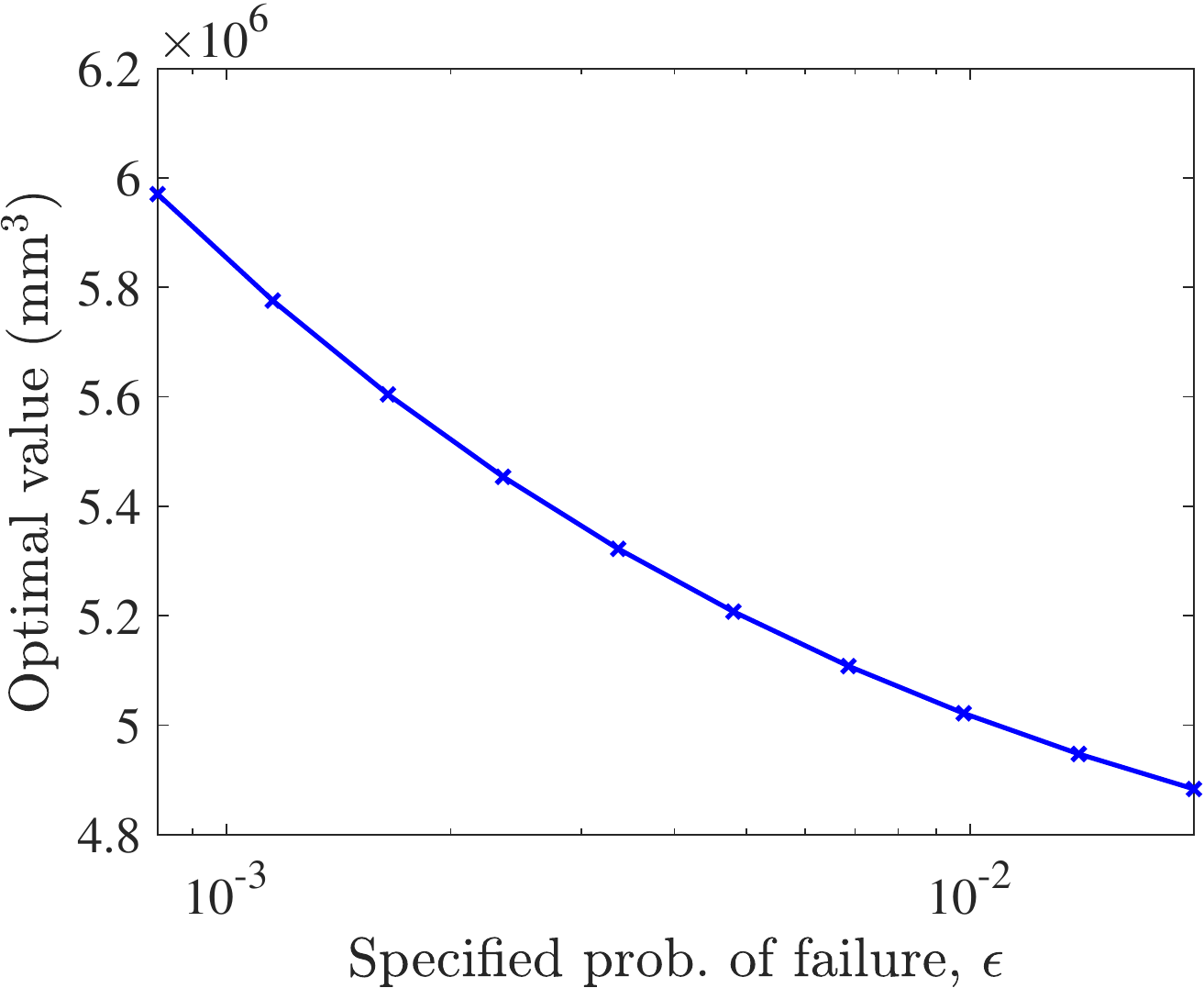}
  }
  \hfill
  \subfloat[]{
  \label{fig:2bar_ell2_all_dist_uncertainty}
  \includegraphics[scale=0.40]{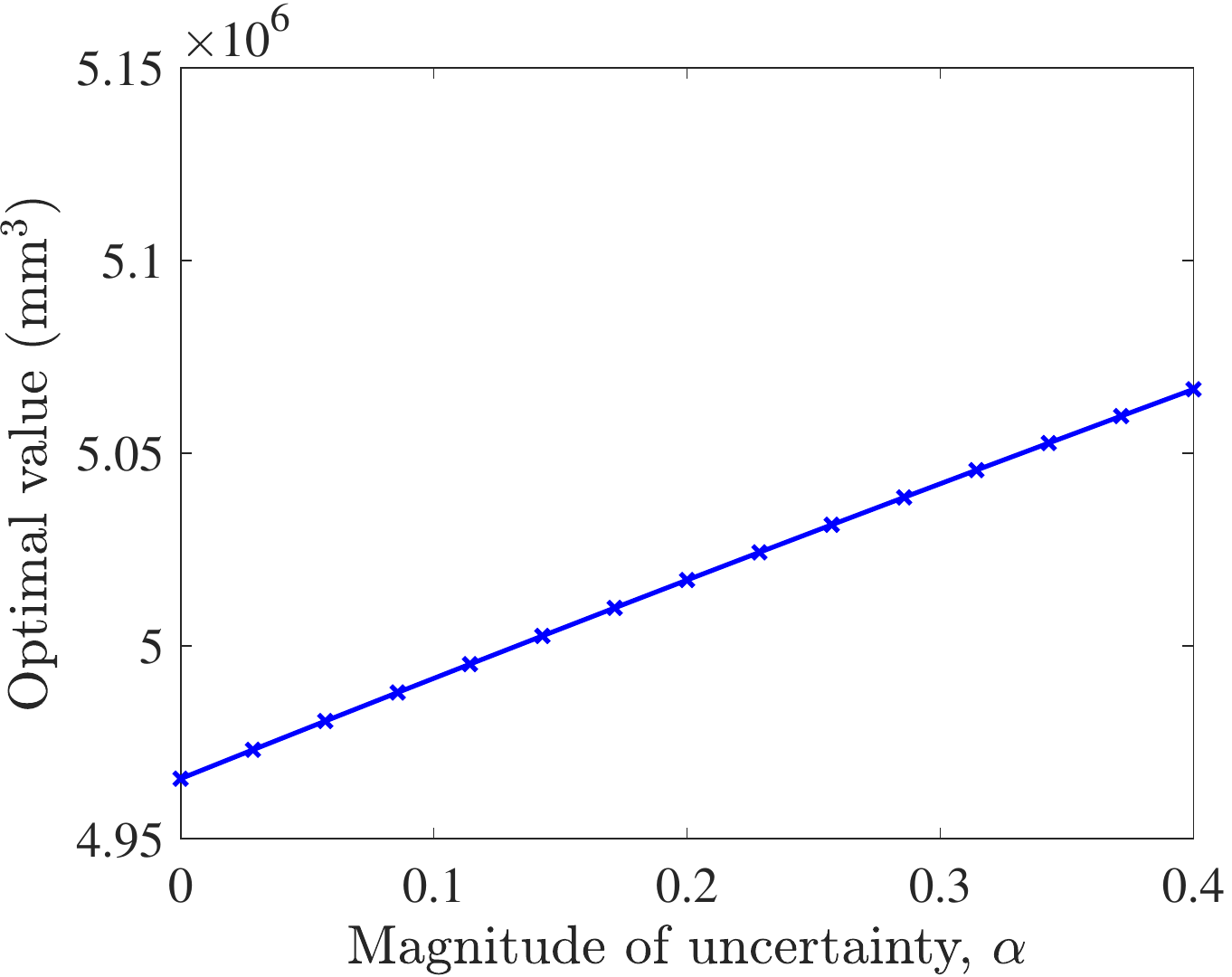}
  }
  \caption[]{Optimal value (example (I); 
  no restriction on distribution type is assumed) versus 
  \subref{fig:2bar_ell1_all_dist_reliability} 
  failure probability (with the $\ell_{\infty}$-norm uncertainty model); 
  \subref{fig:2bar_ell1_all_dist_uncertainty} 
  magnitude of uncertainty ($\ell_{\infty}$-norm); 
  \subref{fig:2bar_ell2_all_dist_reliability} 
  failure probability (with the $\ell_{2}$-norm uncertainty model); and 
  \subref{fig:2bar_ell2_all_dist_uncertainty} 
  magnitude of uncertainty ($\ell_{2}$-norm).    }
  \label{fig:2bar_ell2_all_dist}
\end{figure*}

\subsection{Example (I): 2-bar truss}

Consider a plane truss depicted in \reffig{fig:gs2bar}. 
The truss has $n=2$ members and $d=2$ degrees of freedom of the nodal 
displacements. 
The elastic modulus of the members is $20\,\mathrm{GPa}$. 
A vertical external force of $100\,\mathrm{kN}$ is applied at 
the free node. 
The upper bound for the compliance is $\bar{\pi}=100\,\mathrm{J}$. 

We first consider the uncertainty model with the $\ell_{\infty}$-norm, 
studied in section~\ref{sec:robust.ell.infty}. 
In the uncertainty model in \eqref{eq:uncertainty.set.ell.infty.1} 
and \eqref{eq:uncertainty.set.ell.infty.2}, we put 
$A=B=I$ with $m=k=n$, as considered in \refex{ex:ell.infty.uncertainty}. 
The best estimates, or the nominal values, of $\bi{\mu}$ and $\varSigma$ 
are set to 
\begin{align*}
  \tilde{\bi{\mu}}=\bi{0}, 
  \quad
  \tilde{\varSigma} = 
  \begin{bmatrix}
    0.07 & 0.02 \\ 0.02 & 0.07 \\
  \end{bmatrix}
  . 
\end{align*}
The magnitude of uncertainty is $\alpha=0.2$ and $\beta=0.01$. 
The specified upper bound for the failure probability is $\epsilon=0.01$. 
The optimal solution obtained by the proposed method is listed in the 
row ``$\ell_{\infty}$-norm unc.''\ of \reftab{tab:solution.ex.I}, where 
``obj.\ val.''\ means the objective value at the obtained solution. 
For comparison, the optimal solution of the nominal optimization problem 
(i.e., the conventional structural volume minimization under the 
compliance constraint without considering uncertainty) is also listed. 

The optimization result was verified as follows. 
We randomly generate $\bi{\mu} \in U_{\bi{\mu}}$  and 
$\varSigma \in U_{\varSigma}$, and then generate $10^{6}$ samples 
drawn as $\bi{\zeta} \sim \NC(\bi{\mu},\varSigma)$. 
\reffig{fig:2bar_ell1_x1} and \reffig{fig:2bar_ell1_x2} show the samples 
of $\bi{x}=\tilde{\bi{x}}+\bi{\zeta}$ generated in this manner. 
\reffig{fig:2bar_ell1_sample_app_compl} shows the values of the linearly 
approximated constraint function, 
\begin{align*}
  g(\tilde{\bi{x}}) 
  + \nabla g(\tilde{\bi{x}})^{\top}\bi{\zeta} 
  = \pi(\tilde{\bi{x}}) 
  + \nabla \pi(\tilde{\bi{x}})^{\top} \bi{\zeta} - \bar{\pi} , 
\end{align*}
for these samples. 
Therefore, the ratio of the number of samples of which these function 
values are positive 
divided by the number of all samples (i.e., $10^{6}$) 
should be no greater than $\epsilon$ $(=0.01)$. 
We computed this ratio for each of $10^{4}$ randomly generated samples 
of $\bi{\mu} \in U_{\bi{\mu}}$ and $\varSigma \in U_{\varSigma}$, 
where the continuous uniform distribution was used to generate 
samples of the components of $\bi{\mu}$ and $\varSigma$. 
\reffig{fig:2bar_ell1_prob_of_failure} shows the histogram of 
the values of this ratio computed in this manner, i.e., it shows 
distribution of the failure probability estimated by 
double-loop Monte Carlo simulation. 
It is observed in \reffig{fig:2bar_ell1_prob_of_failure} that, 
for every one of $10^{4}$ probability distribution samples, 
the failure probability is no greater than $\epsilon$. 
Thus, it is verified that the obtained solution satisfies the 
distributionally-robust reliability constraint in 
\eqref{eq:robust.reliable}. 
Indeed, among these samples of the failure probability, the maximum 
value is $0.009054$ $(<\epsilon)$. 
For reference, \reffig{fig:2bar_ell1_exact_failure_prob} shows the 
histogram of failure probabilities computed for the constraint function 
values without applying the linear approximation, i.e., 
$g(\bi{x}) = \pi(\bi{x}) - \bar{\pi}$. 
It is observed in \reffig{fig:2bar_ell1_exact_failure_prob} that, in 
only rare cases, the failure probability exceeds the target value 
$\epsilon$ $(=0.01)$. 
\reffig{fig:2bar_ell1_loop_reliability} shows the variation of the 
optimal value with respect to the upper bound for the failure 
probability $\epsilon$, where $\alpha=0.2$ and $\beta=0.01$ are fixed. 
As $\epsilon$ decreases, the optimal value increases. 
In contrast, \reffig{fig:2bar_ell1_loop_uncertainty} shows the variation 
of the optimal value with respect to $\alpha$ and $\beta$, where 
$\epsilon=0.01$ is fixed. 
Although in \reffig{fig:2bar_ell1_loop_uncertainty} only values of 
$\alpha$ are shown, values of $\beta \in [0,0.02]$ are also varied in a 
manner proportional to $\alpha$. 
As the magnitude of uncertainty increases, the optimal value increases. 

We next consider the uncertainty model with the $\ell_{2}$-norm, 
studied in section~\ref{sec:robust.ell.2}. 
The uncertainty set is defined with $A$, $B$, $\tilde{\bi{\mu}}$, 
$\tilde{\varSigma}$, $\alpha$, and $\beta$ used above. 
The specified upper bound for the failure probability is $\epsilon=0.01$. 
The obtained optimal solution is listed in the row 
``$\ell_{2}$-norm unc.''\ of \reftab{tab:solution.ex.I}. 
It can be observed that the objective value is small compared with the 
solution with the $\ell_{\infty}$-norm uncertainty model. 
This is natural, because, with the common values of $\alpha$ and $\beta$, 
the uncertainty set with the $\ell_{2}$-norm is included in 
the uncertainty set with the $\ell_{\infty}$-norm. 
The optimization result is verified in the same manner as above. 
Namely, \reffig{fig:2bar_ell2_prob_of_failure} shows $10^{4}$ samples of 
the failure probability, each of which was computed with $10^{6}$ 
samples of $\bi{\zeta}$. 
Among these samples, the maximum failure probability is 
$0.009850$ $(< \epsilon)$, which verifies that the obtained solution 
satisfies distributionally-robust reliability 
constraint \eqref{eq:robust.reliable}. 
\reffig{fig:2bar_ell2_loop_reliability} and 
\reffig{fig:2bar_ell2_loop_uncertainty} show the variations of the 
optimal value with respect to the failure probability, $\epsilon$, and 
the magnitude of uncertainty, $\alpha$ and $\beta$, respectively. 
These variations show trends similar to the ones with the 
$\ell_{\infty}$-norm uncertainty model in 
\reffig{fig:2bar_ell1_loop_reliability} and 
\reffig{fig:2bar_ell1_loop_uncertainty}. 

Finally, as discussed in section~\ref{sec:general.non-Gaussian}, we 
consider, not only the normal distributions, but all the probability 
distributions with $\bi{\mu}$ and $\varSigma$ belonging to the 
uncertainty set. 
That is, the set of possible realizations of probability distributions 
is given by \eqref{eq:set.of.all.distributions}. 
\reffig{fig:2bar_ell2_all_dist}  collects the variations of the optimal 
value with respect to the failure probability, $\epsilon$, and 
the magnitude of uncertainty, $\alpha$ and $\beta$ (in the same manner 
as above, values of $\beta \in [0,0.02]$ are varied in a manner 
proportional to $\alpha$). 
Compared with the results for normal distributions in 
\reffig{fig:2bar_ell1_loop} and \reffig{fig:2bar_ell2_loop}, 
the optimal value in \reffig{fig:2bar_ell2_all_dist} is large, as expected. 
Moreover, as $\epsilon$ decreases, the optimal value in 
\reffig{fig:2bar_ell1_all_dist_reliability} and 
\reffig{fig:2bar_ell2_all_dist_reliability} increases drastically 
compared with the cases in \reffig{fig:2bar_ell1_loop_reliability} and 
\reffig{fig:2bar_ell2_loop_reliability}.

\subsection{Example (II): 29-bar truss}

\begin{figure}[tp]
  \centering
  \includegraphics[scale=0.70]{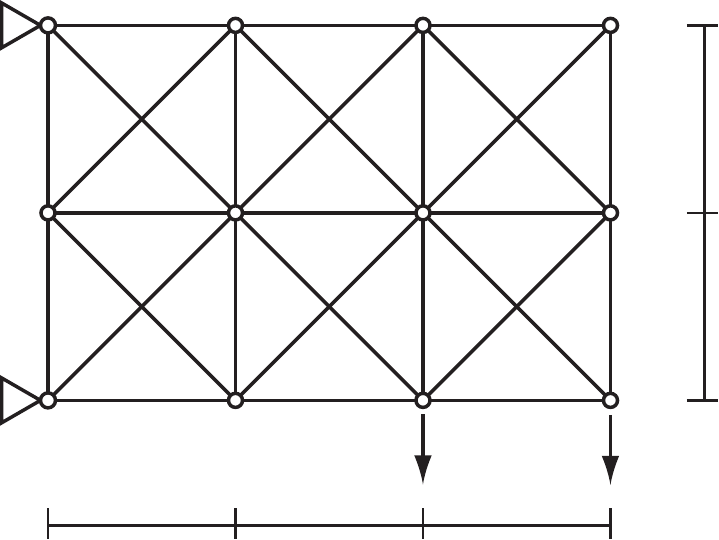}
  \begin{picture}(0,0)
    \put(-172,-38){
    \put(168,82){{\small $1\,\mathrm{m}$}}
    \put(168,120){{\small $1\,\mathrm{m}$}}
    \put(46,31){{\small $1\,\mathrm{m}$}}
    \put(80,31){{\small $1\,\mathrm{m}$}}
    \put(120,31){{\small $1\,\mathrm{m}$}}
    }
  \end{picture}
  \bigskip
  \caption{Problem setting of example (II): 29-bar truss. }
  \label{fig:gs29bar}
\end{figure}

\begin{table*}[tbp]
  \centering
  \caption{Optimal solutions of example (II) 
  with $\epsilon=0.01$, $\alpha=0.2$, and $\beta=0.01$ 
  (probability distributions are assumed to be normal distributions).  }
  \label{tab:solution.ex.II}
  \begin{tabular}{lrr}
    \toprule
    & Obj.\ val.\ ($\mathrm{mm^{3}}$) & $\pi(\tilde{\bi{x}})$ (J) \\
    \midrule
    Nominal optim.\ 
    & $1.6616 \times 10^{7}$ & 1000.00 \\
    $\ell_{\infty}$-norm unc.\
    & $1.7918 \times 10^{7}$ & 917.66 \\
    $\ell_{2}$-norm unc.\
    & $1.7475 \times 10^{7}$ & 944.21 \\
    \bottomrule
  \end{tabular}
\end{table*}

\begin{figure*}[tbp]
  \centering
  \subfloat[]{
  \label{fig:prog2_ell2_loop_reliability}
  \includegraphics[scale=0.40]{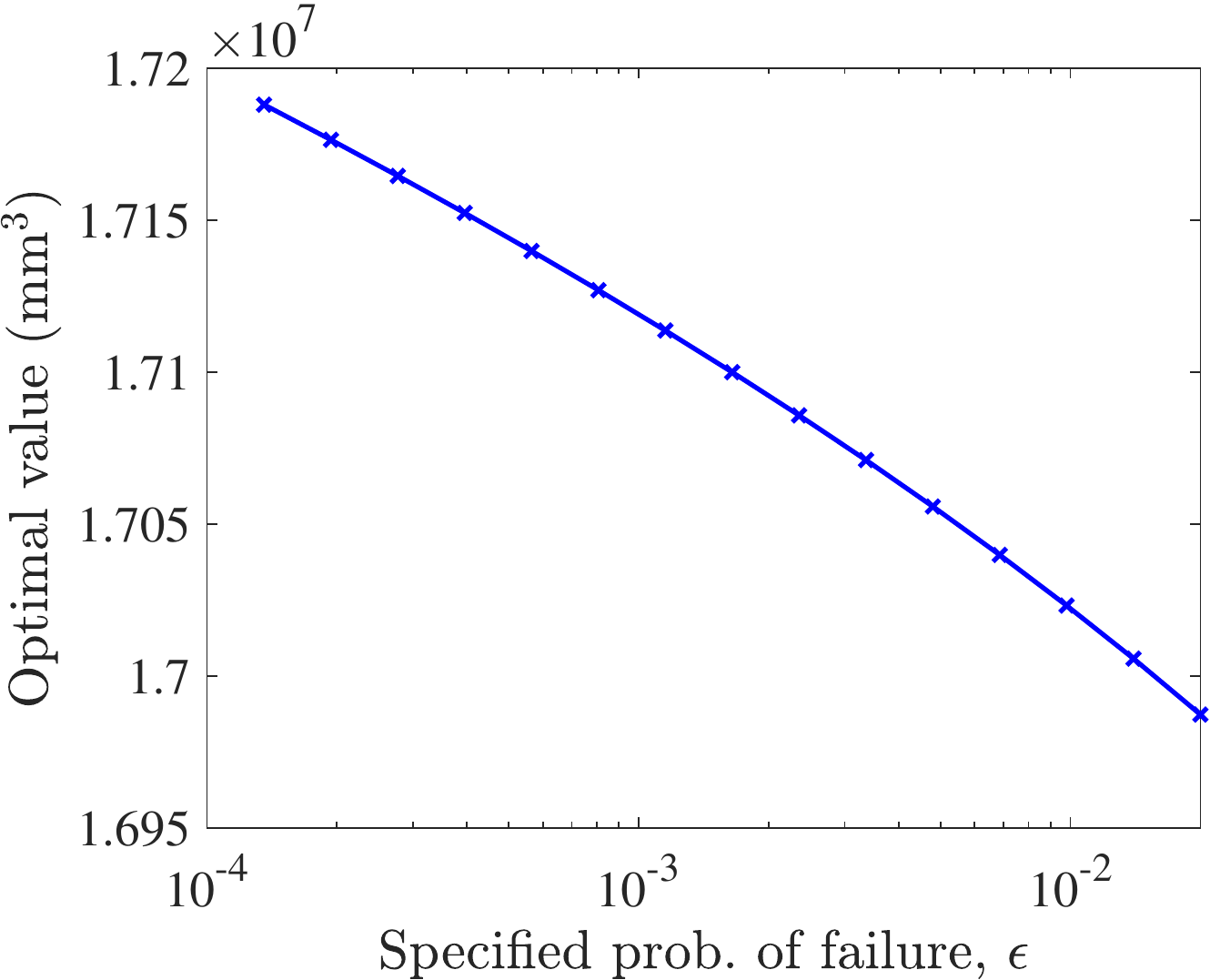}
  }
  \hfill
  \subfloat[]{
  \label{fig:prog2_ell2_loop_uncertainty}
  \includegraphics[scale=0.40]{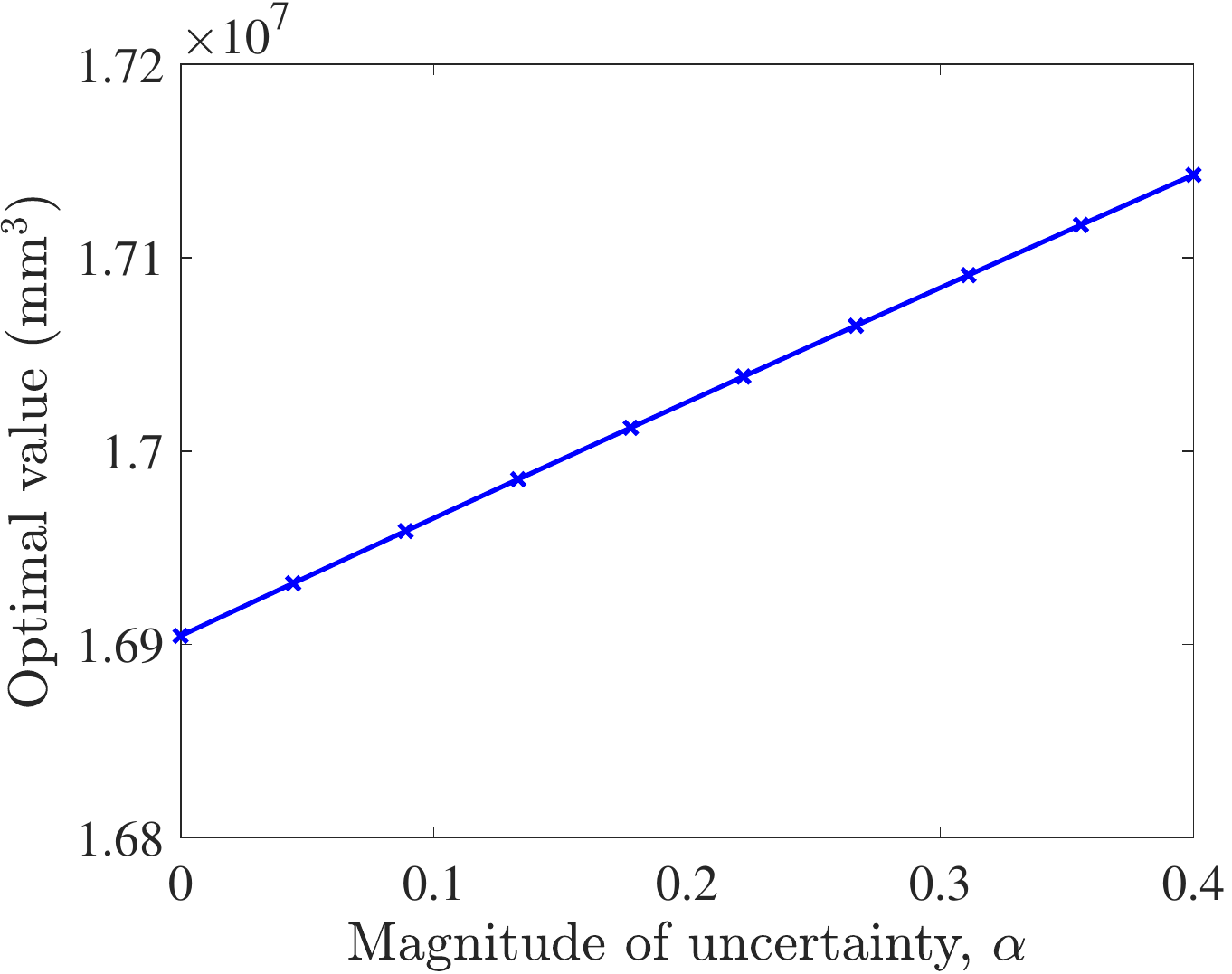}
  }
  \caption[]{Optimal value (example (II) with the $\ell_{2}$-norm 
  uncertainty model; probability distributions are assumed to be normal 
  distributions) versus 
  \subref{fig:prog2_ell2_loop_reliability} failure probability; and 
  \subref{fig:prog2_ell2_loop_uncertainty} magnitude of uncertainty.   }
\end{figure*}

\begin{figure*}[tbp]
  \centering
  \subfloat[]{
  \label{fig:prog2ell2_all_dist_reliability}
  \includegraphics[scale=0.40]{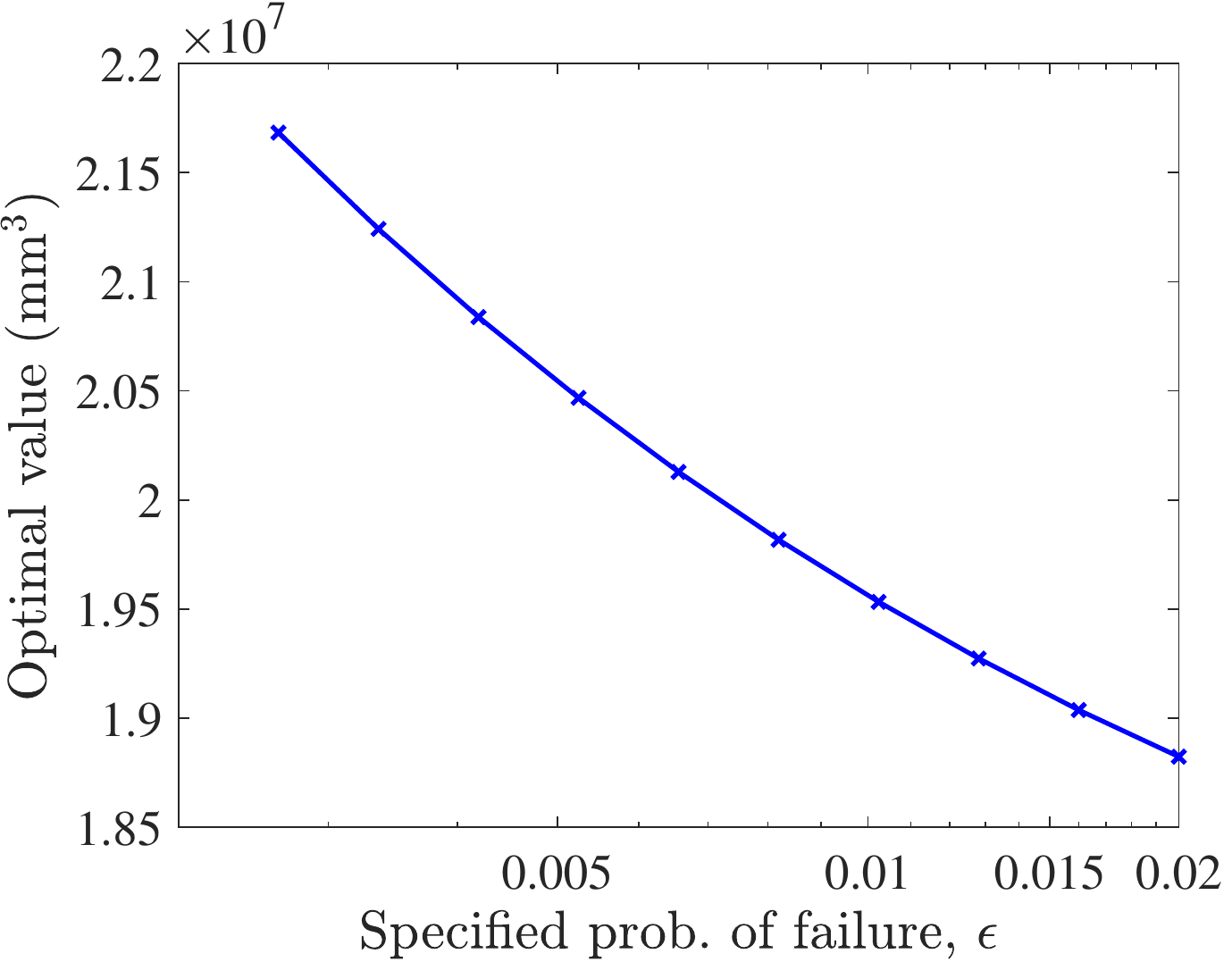}
  }
  \hfill
  \subfloat[]{
  \label{fig:prog2ell2_all_dist_uncertainty}
  \includegraphics[scale=0.40]{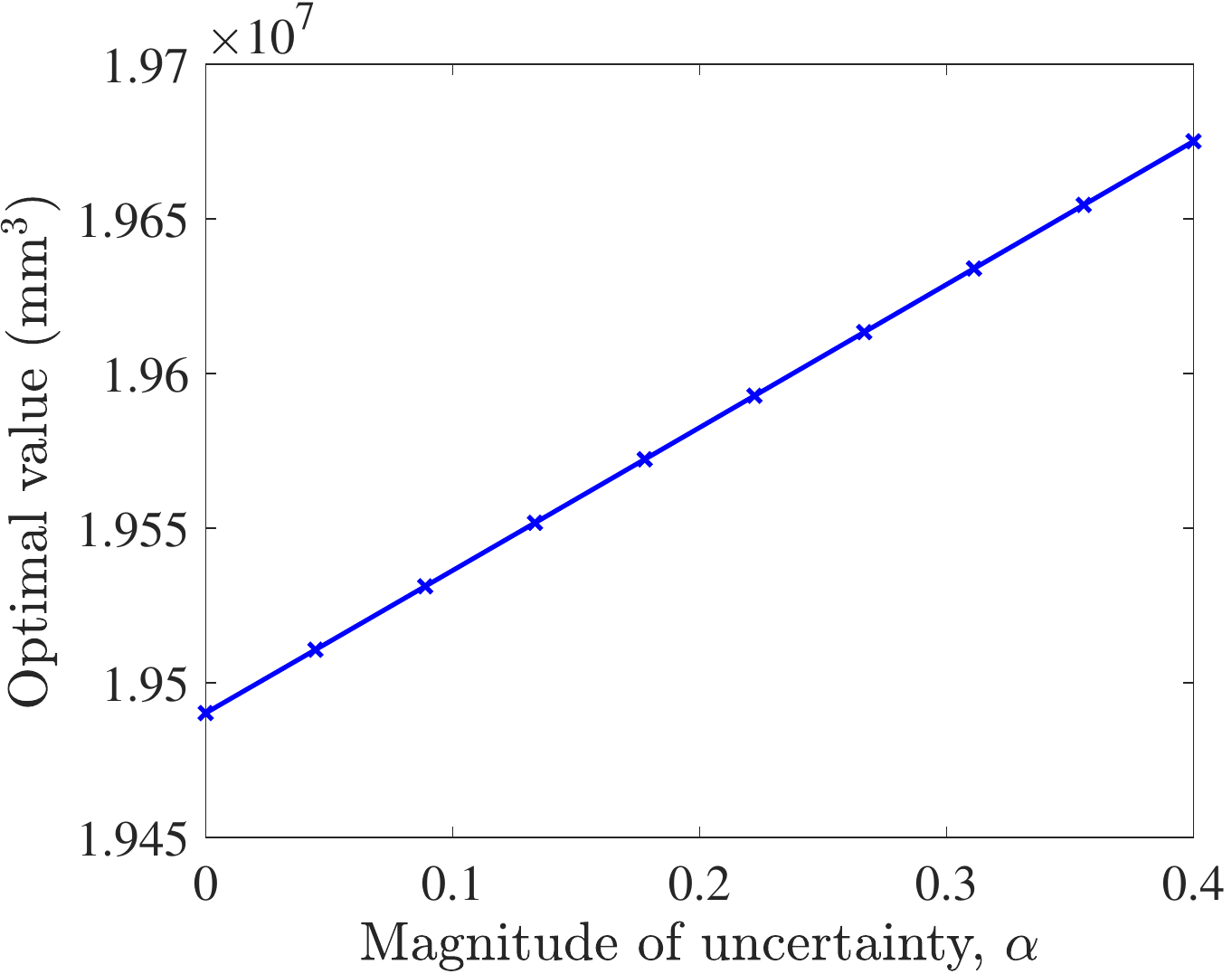}
  }
  \caption[]{Optimal value (example (II) with the $\ell_{2}$-norm 
  uncertainty model; no restriction on distribution type is assumed) 
  versus 
  \subref{fig:prog2ell2_all_dist_reliability} failure probability; and 
  \subref{fig:prog2ell2_all_dist_uncertainty} magnitude of uncertainty.   }
\end{figure*}

\begin{figure*}[tbp]
  \centering
  \subfloat[]{
  \label{fig:prog2_nominal}
  \includegraphics[scale=0.40]{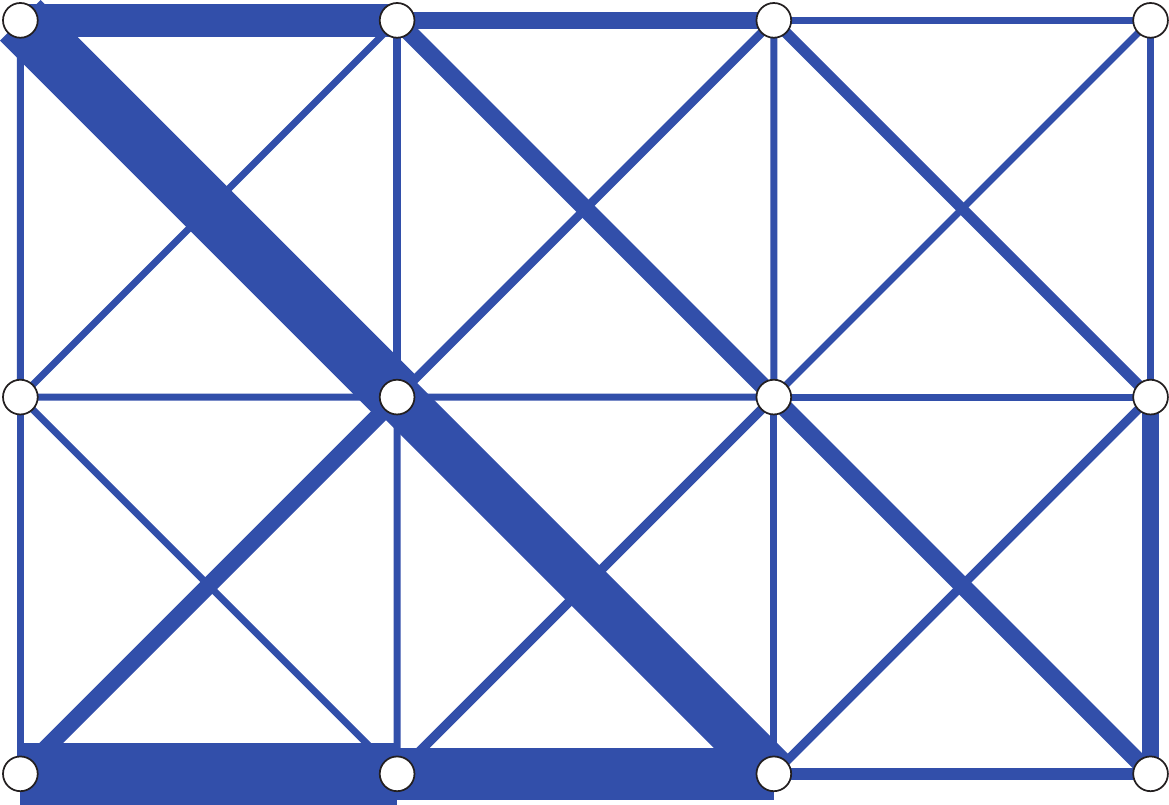}
  }
  \par
  \subfloat[]{
  \label{fig:prog2_ell1_design}
  \includegraphics[scale=0.40]{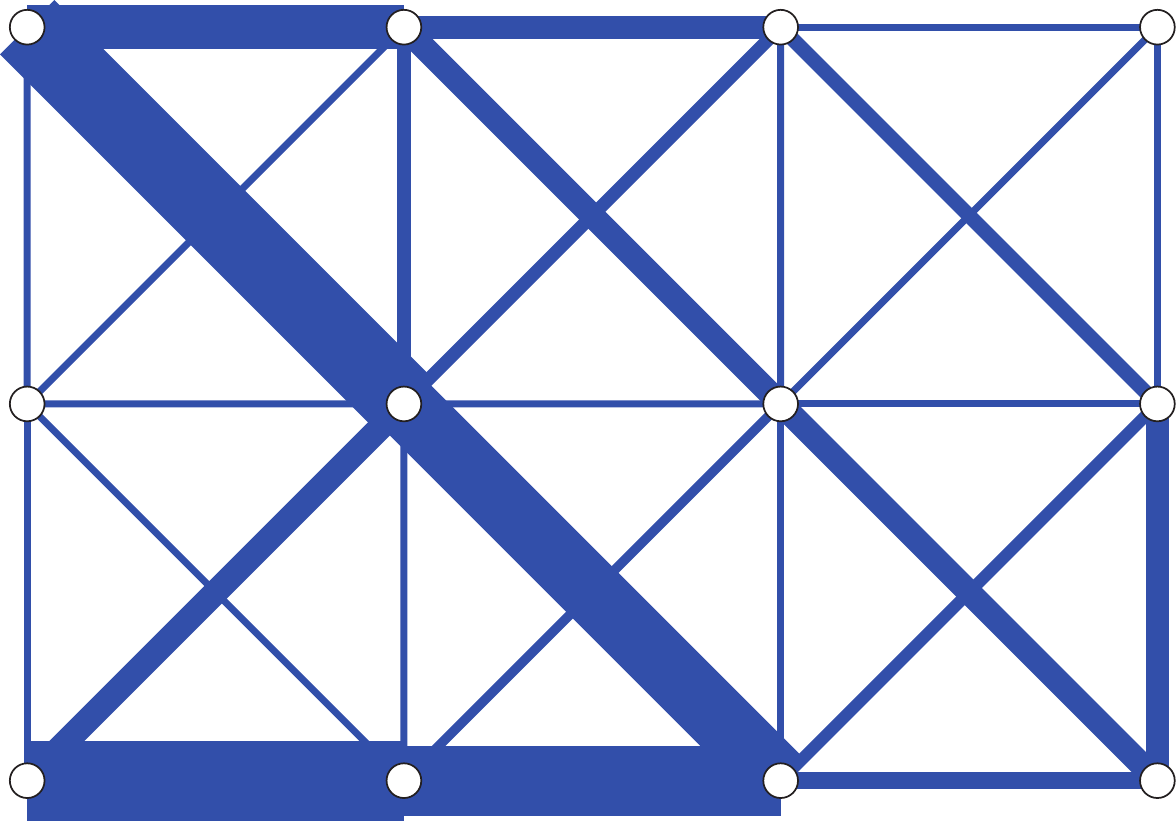}
  }
  \hfill
  \subfloat[]{
  \label{fig:prog2_ell2_design}
  \includegraphics[scale=0.40]{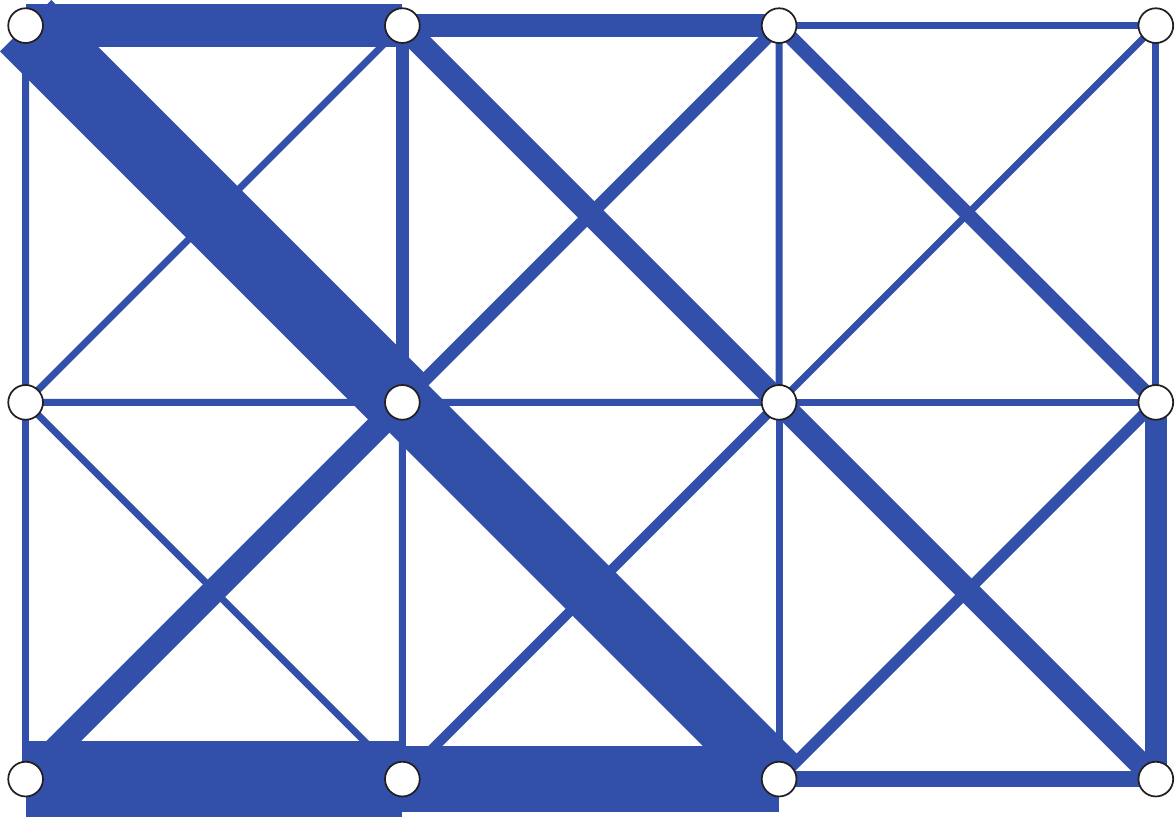}
  }
  \caption[]{Obtained designs of example (II). Optimal solutions 
  \subref{fig:prog2_nominal} without considering uncertainty; 
  \subref{fig:prog2_ell1_design} 
  with the $\ell_{\infty}$-norm uncertainty model; and 
  \subref{fig:prog2_ell2_design} 
  with the $\ell_{2}$-norm uncertainty model. 
  In \subref{fig:prog2_ell1_design} and \subref{fig:prog2_ell2_design}, 
  no restriction on distribution type is assumed, 
  $\epsilon=0.005$, $\alpha=0.2$, and $\beta=0.01$.   }
  \label{fig:prog2_design}
\end{figure*}

Consider a plane truss depicted in \reffig{fig:gs29bar}, where $n=29$ 
and $d=20$. 
The elastic modulus of the members is $20\,\mathrm{GPa}$. 
Vertical external forces of $100\,\mathrm{kN}$ are applied at two nodes 
as shown in \reffig{fig:gs29bar}. 
The upper bound for the compliance is $\bar{\pi}=1000\,\mathrm{J}$. 
The lower bounds for the member cross-sectional areas are 
$\bar{x}_{j} = 200\,\mathrm{mm^{2}}$ $(j=1,\dots,n)$. 

As for the uncertainty model, we consider both the model with the 
$\ell_{2}$-norm, putting $A=B=I$ with $m=k=n$. 
The best estimates of $\bi{\mu}$ and $\varSigma$ are 
\begin{align*}
  \tilde{\bi{\mu}}=\bi{0} , 
  \quad
  \tilde{\varSigma} = 0.05 I + 0.02 (\bi{1} \bi{1}^{\top}) , 
\end{align*}
where $\bi{1} \in \Re^{n}$ is an all-ones column vector. 
The magnitude of uncertainty is $\alpha=0.2$ and $\beta=0.01$. 
The specified upper bound for the failure probability is $\epsilon=0.01$. 

The optimization results obtained by the proposed method are listed in 
\reftab{tab:solution.ex.II}. 
\reffig{fig:prog2_ell2_loop_reliability} and 
\reffig{fig:prog2_ell2_loop_uncertainty} show the variations of the 
optimal value with respect to the failure probability and the magnitude 
of uncertainty, respectively. 

As done in section~\ref{sec:general.non-Gaussian}, we next require that 
the reliability constraint should be satisfied for all the probability 
distributions satisfying $\bi{\mu} \in U_{\bi{\mu}}$ and 
$\varSigma \in U_{\varSigma}$, i.e., for any probability distribution 
belonging to $\PC$ in \eqref{eq:set.of.all.distributions}. 
For the $\ell_{2}$-norm uncertainty, 
\reffig{fig:prog2ell2_all_dist_reliability} and 
\reffig{fig:prog2ell2_all_dist_uncertainty} report the variations of the 
optimal value with respect to the failure probability and the magnitude 
of uncertainty, respectively. 
\reffig{fig:prog2_design} collects the optimal solutions of the 
optimization problem without uncertainty, as well as the 
distributionally-robust RBDO problems with the two uncertainty models. 
Here, the width of each member in the figures are proportional to its 
cross-sectional area.

\section{Conclusions}
\label{sec:conclusion}

This paper has dealt with reliability-based design optimization (RBDO) 
of structures, in which knowledge of the input distribution that the 
design variables follow is imprecise. 
Specifically, we only know that the expected value vector and the 
variance-covariance matrix of the input distribution belong to a 
specified convex set, and do not know their true values. 
Then we attempt to optimize a structure, under the constraint that, 
even for the worst-case input distribution, the failure probability of 
the structure is no greater than the specified value. 
This constraint, called the distributionally-robust reliability 
constraint, is equivalent to infinitely many reliability constraints 
corresponding to all possible realizations of the input distribution. 
Provided that change of a constraint function value is well 
approximated as a linear function of uncertain perturbations of the 
design variables, 
this paper has presented a tractable reformulation of the 
distributionally-robust reliability constraint. 

This paper has established the concept of distributionally-robust RBDO, 
and developed fundamental results. 
Much remains to be studied. 
For instance, in this paper we have considered uncertainty only in the design 
variables. 
Other sources of uncertainty in structural optimization can be explored. 
Also, as discussed in section~\ref{sec:general.multiple.constraints}, 
multiple performance requirement in the form of 
\eqref{eq:multiple.constraint.2} remains to be studied. 
Extension to topology optimization is of great interest. 
Moreover, this paper relies on the assumption that quantity of interest 
is approximated, with sufficient accuracy, as a linear function of 
uncertainty perturbations of the design variables. 
Extension to nonlinear cases can be attempted. 
Finally, development of a more efficient algorithm for solving the 
optimization problem presented in this paper can be studied.

\paragraph{Acknowledgments}

This work is supported by 
Research Grant from the Maeda Engineering Foundation 
and 
JSPS KAKENHI (17K06633, 21K04351).


\begin{thebibliography}{99}
\bibitem[\protect\citeauthoryear{Anjos and Lasserre}{2012}]{AL12}
  {M.~F.~Anjos, J.~B.~Lasserre (eds.)}:
  {\em Handbook on Semidefinite, Conic and Polynomial Optimization}.
  Springer, New York (2012).

\bibitem[\protect\citeauthoryear{Aoues and Chateauneuf}{2010}]{AC10}
  {Y.~Aoues, A.~Chateauneuf}:
  {Benchmark study of numerical methods for reliability-based 
    design optimization}.
  {\em Structural and Multidisciplinary Optimization},
  \textbf{41}, 277--294 (2010).

\bibitem[\protect\citeauthoryear{Ben-Haim}{2006}]{Bh06}
  {Y.~Ben-Haim}:
  {\em Information-gap Decision Theory: Decisions under Severe 
    Uncertainty (2nd ed.)}.
  Academic Press, London (2006).

\bibitem[\protect\citeauthoryear{Ben-Tal {\em et al.\/}}{2009}]{BtEN09}
  {A.~Ben-Tal, L.~El~Ghaoui, A.~Nemirovski}:
  {\em Robust Optimization}.
  Princeton University Press, Princeton (2009).

\bibitem[\protect\citeauthoryear{Ben-Tal and Nemirovski}{1997}]{BtN97}
  {A.~Ben-Tal, A.~Nemirovski}:
  {Robust truss topology optimization via semidefinite programming}.
  {\em SIAM Journal on Optimization},
  \textbf{7}, 991--1016 (1997).

\bibitem[\protect\citeauthoryear{Beyer and Sendhoff}{2007}]{BS07}
  {H.-G.~Beyer, B.~Sendhoff}:
  {Robust optimization---a comprehensive survey}.
  {\em Computer Methods in Applied Mechanics and Engineering},
  \textbf{196}, 3190--3218 (2007).

\bibitem[\protect\citeauthoryear{Boyd and Vandenberghe}{2004}]{BV04}
  {S.~Boyd, L.~Vandenberghe}:
  {\em Convex Optimization}.
  Cambridge University Press, Cambridge (2004).

\bibitem[\protect\citeauthoryear{Calafiore and El Ghaoui}{2014}]{CEg14}
  {G.~C.~Calafiore, L.~El Ghaoui}:
  {\em Optimization Models}.
  Cambridge University Press, Cambridge (2014).

\bibitem[\protect\citeauthoryear{Cho {\em et al.\/}}{2016}]{CCGLLG16}
  {H.~Cho, K.~K.~Choi, N.~J.~Gaul, I.~Lee, D.~Lamb, D.~Gorsich}:
  {Conservative reliability-based design optimization method 
    with insufficient input data}.
  {\em Structural and Multidisciplinary Optimization},
  \textbf{54}, 1609--1630 (2016).

\bibitem[\protect\citeauthoryear{Choi {\em et al.\/}}{2010}]{CAW10}
  {J.~Choi, D.~An, J.~Won}:
  {Bayesian approach for structural reliability analysis and optimization 
    using the Kriging dimension reduction method}.
  {\em Journal of Mechanical Design},
  \textbf{132}, 051003 (2010).

\bibitem[\protect\citeauthoryear{Delage and Ye}{2010}]{DY10}
  {E.~Delage, Y.~Ye}:
  {Distributionally robust optimization under moment uncertainty 
    with application to data-driven problems}.
  {\em Operations Research},
  \textbf{58}, 595--612 (2010).

\bibitem[\protect\citeauthoryear{El Ghaoui {\em et al.\/}}{2003}]{EOO03}
  {L.~El Ghaoui, M.~Oks, F.~Oustry}:
  {Worst-case value-at-risk and robust portfolio optimization: 
    a conic programming approach}.
  {\em Operations Research},
  \textbf{51}, 543--556 (2003).

\bibitem[\protect\citeauthoryear{Grant and Boyd}{2008}]{GB08}
  {M.~Grant, S.~Boyd}:
  {Graph implementations for nonsmooth convex programs}.
  In: V.~Blondel, S.~Boyd, H.~Kimura (eds.),
  {\em Recent Advances in Learning and Control (A Tribute to M.~Vidyasagar)},
  Springer, pp.~95--110 (2008).

\bibitem[\protect\citeauthoryear{Grant and Boyd}{2021}]{CVX}
  {M.~Grant, S.~Boyd}:
  {\em CVX: Matlab Software for Disciplined Convex Programming}.
  \url{http://cvxr.com/cvx/}
  (Accessed April 2021).

\bibitem[\protect\citeauthoryear{Goh and Sim}{2010}]{GS10}
  {J.~Goh, M.~Sim}:
  {Distributionally robust optimization and its tractable approximations}.
  {\em Operations Research},
  \textbf{58}, 902--917 (2010).

\bibitem[\protect\citeauthoryear{Gunawan and Papalambros}{2006}]{GP06}
  {S.~Gunawan, P.~Y.~Papalambros}:
  {A Bayesian approach to reliability-based optimization 
    with incomplete information}.
  {\em Journal of Mechanical Design},
  \textbf{128}, 909--918 (2006).

\bibitem[\protect\citeauthoryear{Guo {\em et al.\/}}{2009}]{GBZG09}
  {X.~Guo, W.~Bai, W.~Zhang, X.~Gao}:
  {Confidence structural robust design and optimization 
    under stiffness and load uncertainties}.
  {\em Computer Methods in Applied Mechanics and Engineering},
  \textbf{198}, 3378--3399 (2009).

\bibitem[\protect\citeauthoryear{Guo {\em et al.\/}}{2011}]{GDG11}
  {X.~Guo, J.~Du, X.~Gao}:
  {Confidence structural robust optimization by non-linear 
    semidefinite programming-based single-level formulation}.
  {\em International Journal for Numerical Methods in Engineering},
  \textbf{86}, 953--974 (2011).

\bibitem[\protect\citeauthoryear{Holmberg {\em et al.\/}}{2015}]{HTK15}
  {E.~Holmberg, C.-J.~Thore, A.~Klarbring}:
  {Worst-case topology optimization of self-weight loaded 
    structures using semi-definite programming}.
  {\em Structural and Multidisciplinary Optimization},
  \textbf{52}, 915--928 (2015).

\bibitem[\protect\citeauthoryear{Huan {\em et al.\/}}{2019}]{HZFY19}
  {Z.~Huan, G.~Zhenghong, X.~Fang, Z.~Yidian}:
  {Review of robust aerodynamic design optimization for air vehicles}.
  {\em Archives of Computational Methods in Engineering},
  \textbf{26}, 685--732 (2019).

\bibitem[\protect\citeauthoryear{Ito {\em et al.\/}}{2018}]{IKK18}
  {M.~Ito, N.~H.~Kim, N.~Kogiso}:
  {Conservative reliability index for epistemic uncertainty 
    in reliability-based design optimization}.
  {\em Structural and Multidisciplinary Optimization},
  \textbf{57}, 1919--1935 (2018).

\bibitem[\protect\citeauthoryear{Ito and Kogiso}{2016}]{IK16}
  {M.~Ito, N.~Kogiso}:
  {Information uncertainty evaluated by parameter estimation 
    and its effect on reliability-based multiobjective optimization}.
  {\em Journal of Advanced Mechanical Design, Systems, and Manufacturing},
  \textbf{10}, 16-00331 (2016).

\bibitem[\protect\citeauthoryear{Jekel and Haftka}{2020}]{JK20}
  {C.~F.~Jekel, R.~T.~Haftka}:
  {Risk allocation for design optimization with unidentified 
    statistical distributions}.
  {\em AIAA Scitech 2020 Forum},
  Orlando (2020).

\bibitem[\protect\citeauthoryear{Jiang {\em et al.\/}}{2013}]{JCFY13}
  {Z.~Jiang, W.~Chen, Y.~Fu, R.-J.~Yang}:
  {Reliability-based design optimization with model bias and data uncertainty}.
  {\em SAE International Journal of Materials and Manufacturing},
  \textbf{6}, 502--516 (2013).

\bibitem[\protect\citeauthoryear{Jung {\em et al.\/}}{2019}]{JCL19}
  {Y.~Jung, H.~Cho, I.~Lee}:
  {Reliability measure approach for confidence-based design optimization 
    under insufficient input data}.
  {\em Structural and Multidisciplinary Optimization},
  \textbf{60}, 1967--1982 (2019).

\bibitem[\protect\citeauthoryear{Kang and Zhang}{2020}]{KZ16}
  {Z.~Kang, W.~Zhang}:
  {Construction and application of an ellipsoidal convex model 
    using a semi-definite programming formulation from measured data}.
  {\em Computer Methods in Applied Mechanics and Engineering},
  \textbf{300}, 461--489 (2016).

\bibitem[\protect\citeauthoryear{Kanno}{2011}]{Kan11}
  {Y.~Kanno}:
  {\em Nonsmooth Mechanics and Convex Optimization}.
  CRC Press, Boca Raton (2011). 

\bibitem[\protect\citeauthoryear{Kanno}{2018}]{Kan18}
  {Y.~Kanno}:
  {Robust truss topology optimization via semidefinite 
    programming with complementarity constraints: 
    a difference-of-convex programming approach}.
  {\em Computational Optimization and Applications},
  \textbf{71}, 403--433 (2018).

\bibitem[\protect\citeauthoryear{Kanno}{2019}]{Kan19}
  {Y.~Kanno}:
  {A data-driven approach to non-parametric reliability-based 
    design optimization of structures with uncertain load}.
  {\em Structural and Multidisciplinary Optimization},
  \textbf{60}, 83--97 (2019).

\bibitem[\protect\citeauthoryear{Kanno}{2020a}]{Kan20}
  {Y.~Kanno}:
  {Dimensionality reduction enhances data-driven 
    reliability-based design optimizer}.
  {\em Journal of Advanced Mechanical Design, Systems, and Manufacturing},
  \textbf{14}, 19-00200 (2020a).

\bibitem[\protect\citeauthoryear{Kanno}{2020b}]{Kan20b}
  {Y.~Kanno}:
  {On three concepts in robust design optimization: 
    absolute robustness, relative robustness, and less variance}.
  {\em Structural and Multidisciplinary Optimization},
  \textbf{62}, 979--1000 (2020b).

\bibitem[\protect\citeauthoryear{Kanno and Takewaki}{2006}]{KT06}
  {Y.~Kanno, I.~Takewaki}:
  {Sequential semidefinite program for robust truss optimization based 
    on robustness functions associated with stress constraints}.
  {\em Journal of Optimization Theory and Applications},
  \textbf{130}, 265--287 (2006).

\bibitem[\protect\citeauthoryear{Keshtegar and Lee}{2016}]{KL16}
  {B.~Keshtegar, I.~Lee}:
  {Relaxed performance measure approach for reliability-based 
    design optimization}.
  {\em Structural and Multidisciplinary Optimization},
  \textbf{54}, 1439--1454 (2016).

\bibitem[\protect\citeauthoryear{Lee {\em et al.\/}}{2010}]{LCG10}
  {I.~Lee, K.~K.~Choi, D.~Gorsich}:
  {Sensitivity analyses of FORM-based and DRM-based performance measure 
    approach (PMA) for reliability-based design optimization (RBDO)}.
  {\em International Journal for Numerical Methods in Engineering},
  \textbf{82}, 26--46 (2010).

\bibitem[\protect\citeauthoryear{Moon {\em et al.\/}}{2018}]{MCCGLG18}
  {M.-Y.~Moon, H.~Cho, K.~K.~Choi, N.~Gaul, D.~Lamb, D.~Gorsich}:
  {Confidence-based reliability assessment considering 
    limited numbers of both input and output test data}.
  {\em Structural and Multidisciplinary Optimization},
  \textbf{57}, 2027--2043 (2018).

\bibitem[\protect\citeauthoryear{Moustapha and Sudret}{2019}]{MS19}
  {M.~Moustapha, B.~Sudret}:
  {Surrogate-assisted reliability-based design optimization: 
    a survey and a unified modular framework}.
  {\em Structural and Multidisciplinary Optimization},
  \textbf{60}, 2157--2176 (2019).

\bibitem[\protect\citeauthoryear{Noh {\em et al.\/}}{2011a}]{NCLGL11}
  {Y.~Noh, K.~K.~Choi, I.~Lee, D.~Gorsich, D.~Lamb}:
  {Reliability-based design optimization with confidence level 
    under input model uncertainty due to limited test data}.
  {\em Structural and Multidisciplinary Optimization},
  \textbf{43}, 443--458 (2011a).

\bibitem[\protect\citeauthoryear{Noh {\em et al.\/}}{2011b}]{NCLGL11b}
  {Y.~Noh, K.~K.~Choi, I.~Lee, D.~Gorsich, D.~Lamb}:
  {Reliability-based design optimization with confidence level 
    for non-Gaussian distributions using bootstrap method}.
  {\em Journal of Mechanical Design},
  \textbf{133}, 091001 (2011b).

\bibitem[\protect\citeauthoryear{Oberkampf {\em et al.\/}}{2004}]{OHJWF04}
  {W.~L.~Oberkampf, J.~C.~Helton, C.~A.~Joslyn, 
    S.~F.~Wojtkiewicz, S.~Ferson}:
  {Challenge problems: uncertainty in system response given uncertain parameters}.
  {\em Reliability Engineering and System Safety},
  \textbf{85}, 11--19 (2004).


\bibitem[\protect\citeauthoryear{P\'{o}lik}{2005}]{Pol05}
  {I.~P\'{o}lik}:
  {\em Addendum to the SeDuMi User Guide: Version 1.1\/}.
  Technical Report, 
  Advanced Optimization Laboratory, McMaster University, Hamilton (2005).
  \url{http://sedumi.ie.lehigh.edu/sedumi/} 
  (Accessed April 2021).

\bibitem[\protect\citeauthoryear{Sch\"{o}bi and Sudret}{2017}]{SS17}
  {R.~Sch\"{o}bi, B.~Sudret}:
  {Structural reliability analysis for p-boxes using multi-level meta-models}.
  {\em Probabilistic Engineering Mechanics},
  \textbf{48}, 27--38 (2017).

\bibitem[\protect\citeauthoryear{Sturm}{1999}]{Stu99}
  {J.~F.~Sturm}:
  {Using SeDuMi 1.02, a MATLAB toolbox for optimization over symmetric cones}.
  {\em Optimization Methods and Software},
  \textbf{11--12}, 625--653 (1999).

\bibitem[\protect\citeauthoryear{Takezawa {\em et al.\/}}{2011}]{TNKK11}
  {A.~Takezawa, S.~Nii, M.~Kitamura, N.~Kogiso},
  {Topology optimization for worst load conditions based on 
    the eigenvalue analysis of an aggregated linear system},
  {Computer Methods in Applied Mechanics and Engineering},
  \textbf{200}, 2268--2281 (2011).

\bibitem[\protect\citeauthoryear{Thore {\em et al.\/}}{2017}]{THK17}
  {C.-J.~Thore, E.~Holmberg, A.~Klarbring}:
  {A general framework for robust topology optimization under 
    load-uncertainty including stress constraints}.
  {\em Computer Methods in Applied Mechanics and Engineering},
  \textbf{319}, 1--18 (2017).

\bibitem[\protect\citeauthoryear{Valdebenito and Schu\"{e}ller}{2010}]{VS10}
  {M.~A.~Valdebenito, G.~I.~Schu\"{e}ller}:
  {A survey on approaches for reliability-based optimization}.
  {\em Structural and Multidisciplinary Optimization},
  \textbf{42}, 645--663 (2010).

\bibitem[\protect\citeauthoryear{Wang {\em et al.\/}}{2020}]{WHYWG20}
  {Y.~Wang, P.~Hao, H.~Yang, B.~Wang, Q.~Gao}:
  {A confidence-based reliability optimization with single loop 
    strategy and second-order reliability method}.
  {\em Computer Methods in Applied Mechanics and Engineering},
  \textbf{372}, 113436 (2020).

\bibitem[\protect\citeauthoryear{Wiesemann {\em et al.\/}}{2014}]{WKS14}
  {W.~Wiesemann, D.~Kuhn, M.~Sim}:
  {Distributionally robust convex optimization}.
  {\em Operations Research},
  \textbf{62}, 1358--1376 (2014).

\bibitem[\protect\citeauthoryear{Yamashita and Yabe}{2015}]{YY15}
  {H.~Yamashita, H.~Yabe}:
  {A survey of numerical methods for nonlinear semidefinite programming}.
  {\em Journal of the Operations Research Society of Japan},
  \textbf{58}, 24--60 (2015).

\bibitem[\protect\citeauthoryear{Yao {\em et al.\/}}{2011}]{YCLvTG11}
  {W.~Yao, X.~Chen, W.~Luo, M.~van~Tooren, J.~Guo}:
  {Review of uncertainty-based multidisciplinary design optimization 
    methods for aerospace vehicles}.
  {\em Progress in Aerospace Sciences},
  \textbf{47}, 450--479 (2011).

\bibitem[\protect\citeauthoryear{Youn and Wang}{2008}]{YW08}
  {B.~D.~Youn, P.~Wang}:
  {Bayesian reliability-based design optimization 
    using eigenvector dimension reduction (EDR) method}.
  {\em Structural and Multidisciplinary Optimization},
  \textbf{36}, 107--123 (2008).

\bibitem[\protect\citeauthoryear{Zaman and Mahadevan}{2017}]{ZM17}
  {K.~Zaman, S.~Mahadevan}:
  {Reliability-based design optimization of multidisciplinary system 
    under aleatory and epistemic uncertainty}.
  {\em Structural and Multidisciplinary Optimization},
  \textbf{55}, 681--699 (2017).

\bibitem[\protect\citeauthoryear{Zaman {\em et al.\/}}{2011}]{ZRMM11}
  {K.~Zaman, S.~Rangavajhala, M.~P.~McDonald, S.~Mahadevan}:
  {A probabilistic approach for representation of interval uncertainty}.
  {\em Reliability Engineering and System Safety},
  \textbf{96}, 117--130 (2011).

\bibitem[\protect\citeauthoryear{Zhang {\em et al.\/}}{2020}]{ZGXLE20}
  {J.~Zhang, L.~Gao, M.~Xiao, S.~Lee, A.~T.~Eshghi}:
  {An active learning Kriging-assisted method for reliability-based 
    design optimization under distributional probability-box model}.
  {\em Structural and Multidisciplinary Optimization},
  \textbf{62}, 2341--2356 (2020).

\end{thebibliography}
\end{document}